\documentclass[11pt]{article}

\usepackage{comment}
\usepackage{enumerate}
\usepackage{amsthm}
\usepackage{amsmath}
\usepackage{amssymb}
\usepackage{mathtools}
\usepackage{bbold,bbm}
\usepackage{hyperref}
\hypersetup{
    colorlinks,
    citecolor=red,
    filecolor=black,
    linkcolor=red,
    urlcolor=black
}

\usepackage{amsthm}

\usepackage{amsthm}

\providecommand{\keywords}
{
	\textbf{\textit{Keywords and phrases:}}
}
\providecommand{\amssubj}
{
	\textbf{\textit{AMS 2010 subject classification:}}
}

\newtheorem{theorem}{Theorem}[section]
\newtheorem{corollary}[theorem]{Corollary}
\newtheorem{lemma}[theorem]{Lemma}
\newtheorem{definition}[theorem]{Definition}
\newtheorem{remark}[theorem]{Remark}
\newtheorem{proposition}[theorem]{Proposition}

\newcommand{\N}{\ensuremath{\mathbb{N}}}

\newcommand{\norm}[1]{\left\lVert#1\right\rVert}
\newcommand{\normm}[1]{{\left\vert\kern-0.25ex\left\vert\kern-0.25ex\left\vert #1 
    \right\vert\kern-0.25ex\right\vert\kern-0.25ex\right\vert}}
\numberwithin{equation}{section}

\DeclareMathOperator{\Borel}{{\mathfrak B}}
\renewcommand{\d}{{\mathrm d}}
\newcommand{\abs}[1]{\left\lvert #1 \right\rvert}    

\newcommand{\scapro}[2]{\langle #1,#2\rangle}  
\DeclareMathOperator{\R}{{\mathbb R}}
\newcommand{\F}{{\mathcal F}}

\newcommand{\M}{\mathcal M}
\DeclareMathOperator{\Id}{{\mathrm{Id}}}
\newcommand{\drift}{F}
\newcommand{\diff}{G}

\allowdisplaybreaks

\newcommand{\1}{\mathbbm{1}} 
\renewcommand{\phi}{\varphi}
\allowdisplaybreaks

\newcommand{\ducp}{d^{\rm ucp}}
\newcommand{\dem}{d^{\rm em}}

\DeclareMathOperator*{\esssup}{ess\,sup}

\title{Stochastic evolution equations\\ driven by arbitrary cylindrical Lévy processes}

\makeatletter
\renewcommand{\thefootnote}{\fnsymbol{footnote}} 
\makeatother

\author{%
Gergely Bod\'o$^{1}$\thanks{Gergely Bod\'o was supported by the NWO wiskundeclusters grant \#613.009.151.}\and
Sonja Cox$^{1}$\thanks{Sonja Cox (ORCID ID: 0000-0002-9417-1542) was partially supported by the NWO grant VI.Vidi.213.070.}\and
Adam Jakubowski$^{2}$\and
Markus Riedle$^{3}$%
}

\date{%
\small
$^{1}$Korteweg--de Vries Institute for Mathematics, University of Amsterdam, 1090 GE Amsterdam, The Netherlands.\\
$^{2}$Faculty of Mathematics and Computer Science, Nicolaus Copernicus University, Toru\'n, Poland.\\
$^{3}$Department of Mathematics, King's College London, London WC2R 2LS, United Kingdom.
\texttt{markus.riedle@kcl.ac.uk}%
}

\begin{document}
\maketitle

\renewcommand{\thefootnote}{\arabic{footnote}}
\setcounter{footnote}{0}

\begin{abstract}
We establish the first existence and uniqueness result for mild solutions of abstract stochastic evolution equations driven by arbitrary cylindrical Lévy processes in Hilbert spaces. The coefficients are assumed to satisfy  global Lipschitz conditions, and no moment assumptions are imposed on the driving noise. The principal difficulty arises from the fact that cylindrical Lévy processes exist solely in a generalised sense and typically admit no semimartingale or Lévy–Itô decomposition, which precludes the use of classical existence methods. To overcome these obstacles, we develop a pathwise adaptive Euler–Peano approximation scheme based on noise-dependent stopping times and a fixed-point formulation of the mild solution operator. The resulting approach avoids stochastic calculus techniques relying on semimartingale decompositions and provides a robust and flexible framework for treating multiplicative cylindrical Lévy noise in infinite-dimensional systems.
\end{abstract}

\begin{flushleft}
	\amssubj{60H15, 60G20, 47D06, 60G51}
	
\end{flushleft}

\section{Introduction}

Stochastic evolution equations in infinite-dimensional spaces arise naturally in the modelling of complex systems, such as fluid dynamics, quantum fields, and spatially extended phenomena. Random perturbations in these settings are often highly irregular and cannot be represented by genuine processes taking values in the state space. Cylindrical Lévy processes provide a natural framework for such noises: they generalise cylindrical Brownian motion to incorporate jumps and heavy tails, offering a rich class of driving signals for infinite-dimensional dynamics.

However, two intrinsic features of cylindrical Lévy processes make the analysis of stochastic evolution equations  particularly challenging. First, these processes exist only in a generalised sense and do not take values in the underlying space. Second, they typically lack a semimartingale or Lévy–Itô decomposition, which rules out classical stochastic calculus techniques. As a result, standard approaches—such as fixed-point arguments, Galerkin approximations, and variational methods—cannot be applied directly. In the multiplicative-noise setting, even embedding the equation into a larger space to recover a genuine Lévy process fails, a strategy that works for additive noise.

 In this paper, we establish the first existence and uniqueness result for an abstract stochastic evolution equation driven by an arbitrary cylindrical Lévy process, under global Lipschitz assumptions on the coefficients and without imposing moment conditions on the noise. This result goes  beyond previous work, which has been restricted to additive noise, such as Brzeźniak and Zabczyk \cite{brz_zab_10}, and Priola and Zabczyk \cite{priola_zab_11},  or to specific classes of the driving noise such as symmetric $\alpha$-stable cylindrical processes in  Bodó, Týbl, and Riedle \cite{b_t_r_2024spdes}, and Kosmala and Riedle \cite{KR}.

A first systematic study of cylindrical Lévy processes in Hilbert and Banach spaces was carried out by Applebaum and Riedle in \cite{Applebaum-Riedle}. A general integration theory with respect to arbitrary cylindrical Lévy processes was subsequently developed by Jakubowski and Riedle \cite{jakubowski_riedle_2017}, and later refined in joint work of Bodó and Riedle \cite{BR}. Both contributions are based on the decoupling method introduced by Kwapień and Woyczyński \cite{kwapien_woyczynski_1992}. Our approach in the current work makes essential use of this integration theory and requires several additional technical refinements, which are developed in the paper.

Our existence and uniqueness result is obtained via an adaptive Euler–Peano approximation scheme specifically designed for equations driven by cylindrical Lévy noise. The method is formulated in terms of the mild solution operator associated with the equation and can be viewed as a pathwise fixed-point approximation. Starting from the initial condition, the approximation is updated only at noise-dependent stopping times determined by the first exit of the current iterate from a prescribed neighbourhood of its image under the mild solution map. Between successive update times the approximation is frozen, while at each update it is reset according to the evolution prescribed by the equation. This construction, which goes back in spirit to the Euler–Peano approach described by Bichteler \cite[Chapter 5.4]{bichteler_2002} and is inspired by the pathwise solution theory for classical stochastic differential equations developed by Lowther \cite{Lowther}, avoids any reliance on semimartingale or Lévy–Itô decompositions and is therefore well suited to the cylindrical setting.
From a deterministic viewpoint, the resulting scheme can be seen as an infinite-dimensional analogue of the Peano–Carathéodory construction, or equivalently of Amann’s notion of $\epsilon$-approximate solutions for ordinary differential equations; see, for example, \cite[Chapter II.7]{Amann:1990}.

A key ingredient of our analysis is a stochastic Grönwall-type inequality due to Lowther \cite{Lowther}. For completeness, we include a proof of this result in Proposition~\ref{pro.ucp_convergence}. While several variants of stochastic Grönwall lemmas are available in the literature—see, for instance, the recent work \cite{SG24} and the references therein—their application typically relies on an explicit identification of the martingale and finite-variation components of the underlying stochastic integrals. This step is particularly delicate in the present setting, since cylindrical integrators generally do not admit a tractable semimartingale decomposition. In particular, a careful analysis of their jump structure is required, as carried out in \cite{b_t_r_2024spdes}.

The paper is organised as follows. Section~2 introduces the functional-analytic and probabilistic framework, including the space of adapted c\`adl\`ag processes, the metrics used in the analysis, and preparatory continuity results in an \'Emery-type topology. In Section~3, we recall the construction of stochastic integrals with respect to cylindrical L\'evy processes and establish auxiliary results, such as dominated convergence statements and quadratic variation identities in the cylindrical setting. Section~4 is devoted to abstract stochastic evolution equations driven by cylindrical L\'evy noise: we introduce the notion of a mild solution, develop a pathwise adaptive Euler--Peano approximation scheme based on noise-dependent stopping times, and, using a stochastic Gr\"onwall-type inequality, prove existence and uniqueness under standard Lipschitz assumptions.

\section{Semimartingales in Hilbert spaces }

Let $U, V$ and $H$ be separable Hilbert spaces equipped with inner products $\langle \cdot, \cdot \rangle$ and corresponding norms $\norm{\cdot}$. We identify each Hilbert space with its dual. The Borel $\sigma$-algebra on $H$ is denoted by $\Borel(H)$.

We denote by $L(U, H)$ the Banach space of bounded linear operators from $U$ to $H$, equipped with the operator norm $\norm{\cdot}_{U \rightarrow H}$. The subspace of Hilbert-Schmidt operators is denoted by $L_2(U, H)$, with the standard Hilbert-Schmidt norm $\norm{\cdot}_{\rm HS}$. For each $F \in L(U, H)$, we write $F^\dagger$ for the adjoint operator of $F$.

Let $(\Omega, \mathcal{A}, (\mathcal{F}_t)_{t \ge 0}, P)$ be a complete filtered probability space. We denote by $L_P^0(\Omega; H)$ the space of equivalence classes of $H$-valued measurable functions on $\Omega$, equipped with the topology of convergence in probability.

An $H$-valued stochastic process $(X(t): \, t \ge 0)$ is said to have c\`adl\`ag paths if its sample paths are right-continuous on $[0,\infty)$ and have left limits  on $(0, \infty)$. The definition of c\`agl\`ad paths is analogous. For such a process $X$, we define the running supremum process by
\[
X^*(t):=\sup_{s\leq t}\norm{X(s)}.
\] 
Given a stopping time $\tau$, the corresponding stopped process $X^\tau$ is defined by
$X^\tau(t)=X(\tau \wedge t)$ for all $t\ge 0$.

Let $T>0$ be fixed. For $H$-valued c\`adl\`ag processes $X$ and $Y$ define 
\[
 \ducp_T(X,Y):=  E\left[ \sup_{t\in [0,T]} \norm{X(t)-Y(t)}\wedge 1\right]
\] 
By identifying processes which are almost surely indistinguishable, $\ducp_T$ becomes a metric on the space 
of $H$-valued c\`adl\`ag processes on $[0,T]$.

Let $D_H$ denote the space of $H$-valued adapted c\`adl\`ag processes, which is equipped with the topology of  uniform convergence on compacts in probability, which is induced by the pseudo-metric
\[
\ducp (X,Y):=\sum_{k=1}^\infty \frac{1}{2^k} \ducp_k(X,Y). 
\]
By identifying processes which are almost surely indistinguishable, $\ducp$ becomes a metric on $D_H$, turning $(D_H, \ducp)$ into a complete metrisable topological vector space.

We provide three simple observations regarding convergence in $\ducp$. 
\begin{lemma}\label{le.ucp_consequence}
Let $(X_n)_{n \in \mathbb{N}}$ be a sequence of processes in $D_H$ converging to $0$ in $\ducp$. Then there exists a sequence $(\epsilon_n)_{n \in \mathbb{N}}$ of strictly positive numbers converging to $0$ such that, for each $T>0$, 
	\[ \lim_{n \to  \infty}P\left(X_n^*(T)>\epsilon_n\right)=0. \]
\end{lemma}

\begin{proof}
	Let $(a_{n,m})_{n,m\in \N}$ be a doubly indexed $\R$-valued sequence with $\lim_{n\rightarrow \infty} a_{n,m} = 0$ for all $m\in \N$. We claim that there exists a sequence $(m_n)_{n\in \N}$ tending to infinity such that $\lim_{n\rightarrow \infty} a_{n,m_n}=0$. Indeed, by rescaling we can assume $\sup_{n\in\N} | a_{n,1} | \leq 1$.
    By assumption, for all $m\in \N$ there exists an $N_m\in \N$ such that $N_m = \inf\{ n\in \N : \sup_{k\geq n} | a_{k,m} |\leq  \frac{1}{m}\}$. As $N_1=1$, we can define $m_n \in \N$ by $m_n = \sup\{\ell \in \{1,\ldots,n\} : N_\ell \leq n\}$. Note that $(m_n)_{n\in \mathbb{N}}$ is non-decreasing. Moreover, $\lim_{n\rightarrow \infty}m_n= \infty$, since if we had $\sup_{n\in \N} m_n = M$ for some $M \in \mathbb{N}$, then  $N_{M+1}=\infty$, which leads to a contradiction. Finally, $|a_{n,m_n}|\leq \frac{1}{m_n}$ by definition of $m_n$, which confirms our claim. The lemma now follows from applying this observation to $a_{n,m}:= P(X_n^*({m})>m^{-1})$.
\end{proof}

\begin{lemma}\label{le.bounded_in_prob}
	Let $(\xi_n)_{n \in \mathbb{N}}$ be a sequence of $H$-valued random variables such that 
	$$ \lim_{n \rightarrow \infty}\lambda_n \xi_n =0\qquad\text{in }L_P^0(\Omega; H)$$
	 for every real sequence $(\lambda_n)_{n \in \mathbb{N}}$ converging to $0$ as $n\rightarrow \infty$. Then the collection $\{\xi_n:n \in \mathbb{N}\}$ is bounded in probability, i.e., for all $\epsilon>0$ there exists an $m>0$ such that $\sup_{n\in \N} P(\| \xi_n \|>m)<\epsilon$.
\end{lemma}

\begin{proof}
	Suppose that $(\xi_n)_{n\in \N}$ is not bounded in probability. Then there exists an $\epsilon >0$ and a sequence $(n_m)_{m\in \N}$ of integers tending to infinity such that $P(\|\xi_{n_m}\|>m)>\epsilon$. In particular, $(m^{-1} \xi_{n_m})_{m\in \N}$ fails to converge to $0$ in $L_P^0(\Omega; H)$, which contradicts the assumption.
\end{proof}

\begin{lemma}\label{lem:sum_of_processes}
Let $(X_n)_{n\in \N}$ be a sequence of adapted $H$-valued c\`adl\`ag processes on $[0,T]$ such that $ \lim_{n\rightarrow \infty} \ducp_T(X_n,0)=0$. Then there exists a subsequence $(X_{n_k})_{k\in \N}$ such that $\sum_{k\in \N} \| X_{n_k}\|$ is an adapted real-valued c\`adl\`ag process.
\end{lemma}

\begin{proof}
By ucp convergence, there exists a subsequence $(X_{n_k})_{k\in \N}$ such that 
\begin{equation*}P(X_{n_k}^*(T)>2^{-k})<2^{-k},
\end{equation*}
so that the Borel-Cantelli lemma implies that $\sum_{k\in \N} X_{n_k}^*(T)<\infty$ almost surely.

For $\omega$ in this almost sure event, define
	\[
	S_m(t,\omega):=\sum_{k=1}^m \|X_{n_k}(t,\omega)\|,
	\qquad t\in[0,T].
	\]
	Each $S_m(\cdot,\omega)$ is càdlàg, being a finite sum of càdlàg functions. Moreover, for $n\ge m$,
	\[
	\sup_{t\in[0,T]} |S_n(t,\omega)-S_m(t,\omega)|
	\le \sum_{k=m+1}^n X_{n_k}^*(T,\omega).
	\]
	Since $\sum_k X_{n_k}^*(T,\omega)<\infty$, the sequence $(S_m(\cdot,\omega))_{m\in\N}$ is uniformly Cauchy on $[0,T]$. Therefore it converges uniformly to
	$\sum_{k=1}^\infty \|X_{n_k}(t,\omega)\|$ for all $t\in[0,T]$, which completes the proof.
\end{proof}

We introduce another metric on $D_H$. To this end, fix $T>0$ and let  $\mathcal{S}_{{\rm prd}}^{1, {\rm op}}:=\mathcal{S}_{{\rm prd}}^{1, {\rm op}}(H)$ denote the space of all processes $(\Gamma(t):\ t\in [0,T])$ of the form
\begin{align}\label{eq.simple-op-1}
	\Gamma (\omega, t)=\Gamma_0(\omega)\mathbb{1}_{\{0\}}(t)+\sum_{i=1}^{n-1} \left(\sum_{k=1}^{N(i)}O_{i,k}\mathbbm{1}_{A_{i,k}}(\omega)\right) \mathbbm{1}_{ (t_i,t_{i+1}]}(t),
\end{align}
where $\Gamma_0$ is an $\mathcal{F}_0$-measurable $L(H)$-valued random variable, $0=t_1<\cdots < t_{n}=T$, $A_{i,k} \in \mathcal{F}_{t_i}$ and $O_{i,k} \in L(H)$ for all $i=1,...,n-1$ and $k=1,...,N(i)$, such that 
\begin{align*}
	\sup_{(\omega,t)\in \Omega \times [0,T]}\norm{\Gamma(\omega,t)}_{H \rightarrow H}\leq 1.
\end{align*}
The stochastic integral process of a simple process $\Gamma$ of the form \eqref{eq.simple-op-1} with respect to any $H$-valued, c\`adl\`ag process $(X(t):\, t\in [0,T])$ is defined by 
\[
\int_{(0,t]} \Gamma(s)\, \d X(s):= \sum_{i=1}^{n-1}  \left(\sum_{k=1}^{N(i)}O_{i,k}\mathbbm{1}_{A_{i,k}}(\omega)\right)\big(X(t_{i+1}\wedge t)- X(t_i\wedge t)\big). 
\]
We introduce for $H$-valued c\`adl\`ag processes $X$ and $Y$  on $[0,T]$ the mapping 
\[
\dem_T(X,Y):=  \sup_{\Gamma \in \mathcal{S}_{{\rm prd}}^{1, {\rm op}}}E\Bigg[\sup_{t\in [0,T]} 
\norm{\Gamma(0)(X-Y)(0)+ \int_{(0,t]} \Gamma(s)\, \d(X-Y)( s)}\wedge 1 \Bigg]. 
\]
By identifying processes which are almost surely indistinguishable, $\dem_T$ becomes a metric
on the space  of $H$-valued c\`adl\`ag processes on $[0,T]$. Hence, we obtain a metric on $D_H$ by 
\[
\dem (X,Y):=\sum_{k=1}^\infty \frac{1}{2^k} \dem_k(X,Y). 
\]

The metric $\dem$ is defined analogously to the metric associated with the semimartingale topology, also known as the \'Emery topology, on the space of real-valued semimartingales, originally introduced in \cite{Emery}. In contrast to \cite{Emery}, where the metric is defined via a supremum over all measurable, bounded functions,  we construct the metric using the class of simple processes $\mathcal{S}_{{\rm prd}}^{1, {\rm op}}$. This approach not only avoids the non-separability of $L(H)$, the space of bounded linear operators on $H$, but also simplifies the definition of the stochastic integral in the following section. For similar reasons, we restrict attention to deterministic time intervals when defining simple processes, rather than allowing random intervals. In the real-valued case, both choices lead to equivalent metrics; see, for example, \cite{Karandikar-Rao}.

Obviously, by taking $\Gamma(\omega,t)=\Id_H\1_{(0,T]}(t) $ for all $t\in [0,T]$ and $\omega\in\Omega$,  where $\Id_H$ is the identity on $H$, it follows that 
$d_T^{em}(X,Y)\to 0$ implies $d_T^{ucp}(X,Y)\to 0$. Consequently, $\dem$ induces a topology on $D_H$ stronger than $\ducp$.  

The metric $\dem_T$, defined for c\`adl\`ag processes on $[0, T]$, generates the same topology on $D_H$ even when formulated without taking a supremum over time. This observation, established in the following lemma, enables us to relate the present results to those in our earlier work \cite{BR}.

\begin{lemma}\label{le.Emery-ucp}
Let $T>0$. On the space of \(H\)-valued, adapted, c\`adl\`ag processes on \([0,T]\), define a quasi-metric \(\rho_T^{\mathrm{em}}\) by setting, for each pair of processes \(X,Y\) in this space,
\[
\rho_T^{em}(X,Y)=  \sup_{\Gamma \in \mathcal{S}_{{\rm prd}}^{1, {\rm op}}}E\Bigg[
\norm{\Gamma(0)(X-Y)(0)+ \int_{(0,T]} \Gamma(s)\, \d(X-Y)( s)}\wedge 1 \Bigg].
\]
Then $\rho_T^{\rm em}$ and $d_T^{\rm em}$ generate the same topology on the space of $H$-valued, adapted, c\`adl\`ag processes on $[0, T]$.
\end{lemma}

\begin{proof}
Let $(\pi_n)_{n\in\N}$ be a sequence of partitions of $(0,T]$ satisfying $\lim_{n\rightarrow \infty} |\pi_n|=0$. By Fatou's lemma and the fact that the integral process is c\`adl\`ag, it holds for all $X,Y$ in the space of $H$-valued, adapted, c\`adl\`ag processes on $[0,T]$, that 
\begin{align*}
d_T^{\rm em}(X,Y)&= \sup_{\Gamma \in \mathcal{S}_{{\rm prd}}^{1, {\rm op}}}E\Bigg[ \sup_{t\in [0,T]}
\norm{\Gamma(0)(X-Y)(0)+ \int_{(0,t]} \Gamma(s)\, \d(X-Y)( s)}\wedge 1 \Bigg]
\\ & \quad \leq 
\sup_{\Gamma \in \mathcal{S}_{{\rm prd}}^{1, {\rm op}}}
\liminf_{n\rightarrow \infty}
E\Bigg[ \sup_{t\in \pi_n}
\norm{\Gamma(0)(X-Y)(0)+ \int_{(0,t]} \Gamma(s)\, \d(X-Y)( s)}\wedge 1 \Bigg],
\\ & \quad \leq 
\liminf_{n\rightarrow \infty}
\sup_{\Gamma \in \mathcal{S}_{{\rm prd}}^{1, {\rm op}}}E\Bigg[ \sup_{t\in \pi_n}
\norm{\Gamma(0)(X-Y)(0)+ \int_{(0,t]} \Gamma(s)\, \d(X-Y)( s)}\wedge 1 \Bigg],
\end{align*}
whereas the reverse inequality follows trivially. It thus suffices to prove that for every partition $\pi$, the quasi-metric defined for each $X,Y$ in the space of $H$-valued, adapted, c\`adl\`ag processes on $[0,T]$, given by
\begin{align*}
& \sup_{\Gamma \in \mathcal{S}_{{\rm prd}}^{1, {\rm op}}}E\Bigg[ \sup_{t\in \pi}
\norm{\Gamma(0)(X-Y)(0)+ \int_{(0,t]} \Gamma(s)\, \d(X-Y)( s)}\wedge 1 \Bigg],
\end{align*}
generates the same topology as $\rho_T^{\rm em}$ on the space of $H$-valued, adapted, c\`adl\`ag processes on the interval $[0,T]$.
\par
Let $Z_0$ be an $\mathcal{F}_0$-measurable $H$-valued random variable, and $Z_i$ be $\F_{t_{i+1}}$-measurable $H$-valued random variables for some $0\le t_1\le \dots \le t_{n}\le T$ and $i=1,\dots, n-1$. Define the partial sums $S_k:=\sum_{i=0}^k Z_i$ for $k=0,1,\dots,n-1$. 
	Letting
	\[
	A_i:=\Big\{\omega\in\Omega:\ \sup_{k\in\{0,\dots,i-1\}} \norm{S_k(\omega)}< c \Big\},\qquad i=1,\dots,n-1,
	\]
	we obtain $A_i\in \F_{t_i}$ for all $i=1,\dots,n-1$. Let
	\[
	\tau:=\inf\Big\{k\in\{0,1,\dots,n-1\}:\ \norm{S_k}\ge c\Big\},
	\]
	with the convention $\tau=\infty$ if the set is empty. On the event $\{\tau<\infty\}$ we have $\1_{A_i}=1$ for all $i\le \tau$ and $\1_{A_i}=0$ for all $i>\tau$, and therefore
	\[
	\sum_{i=1}^{n-1} \1_{A_i} Z_i = \sum_{i=1}^{\tau} Z_i.
	\]
	In particular,
	\[
	\Big\{\sup_{k\in\{0,1,\dots,n-1\}} \norm{S_k}\ge c \Big\}
	= \{\tau<\infty\}
	\subseteq
	\Big\{ \norm{Z_0+\sum_{i=1}^{n-1} \1_{A_i} Z_i}\ge c \Big\},
	\]
	and hence
	\[
	P\Big( \sup_{k\in\{0,1,\dots, n-1\}} \norm{\sum_{i=0}^k  Z_i}\ge c \Big)
	\le  P\Big( \norm{Z_0+\sum_{i=1}^{n-1} \1_{A_i} Z_i}\ge c  \Big).
	\]
	
	Now let $\Gamma\in \mathcal{S}_{{\rm prd}}^{1, {\rm op}}$; let $0\leq t_1 \leq \ldots \leq t_n \leq T$ and let $\xi_0$ be an $L(H)$-valued $\mathcal{F}_0$-measurable random variable, and $\xi_1,\ldots,\xi_{n-1}$ be $L(H)$-valued $\F_{t_i}$-measurable simple random variables with
	$
	\Gamma(\omega,t) = \xi_0\mathbb{1}_{0}(t)+\sum_{i=1}^{n-1} \xi_i(\omega)\, \mathbbm{1}_{(t_i,t_{i+1}]}(t).
	$
	For any $H$-valued, adapted, c\`adl\`ag processes $X$ and $Y$ on $[0, T]$, define $Z_0:=\xi_0(X-Y)(0)$, and
	\[
	Z_i:= \xi_i\big((X-Y)(t_{i+1})-(X-Y)(t_i)\big),\qquad i=1,\dots,n-1.
	\]
	Since $A_i\in \F_{t_i}$, the process
	$
	\tilde{\Gamma}(\omega,t)
	= \xi_0(\omega)\mathbb{1}_{0}(t)+
	\sum_{i=1}^{n-1} \xi_i(\omega)\, \1_{A_i}(\omega)\, \1_{(t_i,t_{i+1}]}(t)
	$
	belongs to $\mathcal{S}_{{\rm prd}}^{1, {\rm op}}$, and we have
	\[
	\sum_{i=1}^{n-1} \1_{A_i} Z_i
	=
	\int_{(0,T]} \tilde{\Gamma}(s)\, d(X-Y)(s).
	\]
	Applying the inequality above yields the desired conclusion by the equivalent characterization of the topology of convergence in probability.
\end{proof}

An adapted, $H$-valued stochastic process $(X(t):\, t\ge 0)$ with c\`adl\`ag trajectories is called a semimartingale if it admits a decomposition $X=M+A$, where $M$ is a locally square-integrable, $H$-valued martingale, and $A$ is a stochastic process in $H$ with paths of locally finite variation $P$-almost surely. 
Semimartingales in Hilbert spaces are studied in detail in \cite{met82}, to which we refer for the results used in this work. In particular, stochastic integration with respect to $H$-valued semimartingales is understood in the sense introduced in \cite{met82}. 

\begin{lemma}\label{le.cont_of_semimartingales}
    Let $H$ and $V$ be separable Hilbert spaces,  $B \in L(H,V)$, and let $X$ be an $H$-valued semimartingale. 
    Then, for every $0\le p\le q$ and $S\in \F_p$, we have 
         \[d_T^{em}\left(\1_S \1_{(p,q]}BX,0\right)\leq {\max\{\norm{B}_{H\to V}, 1\} }\, d_T^{em}\big(X,0\big).\]
\end{lemma}

\begin{proof}
As $B\in L(H,V)$, we can assume, by taking $V=\operatorname{Ran}(B)$, that $\operatorname{dim}(V)\leq \operatorname{dim}(H)$.
Let $i:V\rightarrow H$ be an isometric embedding and set $M:=\max\{\norm{B}_{H\to V}, 1\}$.    For each $\Gamma \in \mathcal{S}_{{\rm prd}}^{1, {\rm op}}(V)$, it follows that $i\, \1_S \1_{(p,q]}\Gamma B/M\in \mathcal{S}_{{\rm prd}}^{1, {\rm op} }(H)$. Hence, 
\begin{align*}
    d_T^{em}\big(\1_S \1_{(p,q]}BX,0\big)&=\sup_{\Gamma \in \mathcal{S}_{{\rm prd}}^{1, {\rm op}}(V)}E\Bigg[
\norm{ \int_{(0,T]} \Gamma(s)\, \d\big(\1_S \1_{(p,q]}BX\big)( s)}_V\wedge 1 \Bigg]\\
&=\sup_{\Gamma \in \mathcal{S}_{{\rm prd}}^{1, {\rm op}}(V)}E\Bigg[
\norm{ \int_{(0,T]} \1_{(p,q]}\1_S\Gamma(s) \, \d (BX)(s)}_V\wedge 1 \Bigg]\\
&=\sup_{\Gamma \in \mathcal{S}_{{\rm prd}}^{1, {\rm op}}(V) }E\Bigg[
\norm{ \int_{(0,T]} i\,\1_{(p,q]}\1_S\Gamma(s)B\, \d X(s)}_H\wedge 1 \Bigg]\\
&\leq M\sup_{\Delta \in \mathcal{S}_{{\rm prd}}^{1, {\rm op}}(H)}E\Bigg[
\norm{ \int_{(0,T]} \Delta(s)\, \d X(s)}_H\wedge 1 \Bigg]\\
&\leq\, M\, d_T^{em}(X,0),
\end{align*}
from which our claim immediately follows.
\end{proof}

Let $X$ be an $H$-valued semimartingale. Its (scalar) quadratic variation is defined 
as the real-valued stochastic process $([X](t):t\ge 0)$ given by
\begin{align}\label{def.quad_var}
	[X](t):=\norm{X(t)}^2 -\norm{X(0)}^2-2 \int_{(0,t]} \scapro{X(s-)}{\cdot}\, \d X(s) \qquad\text{for all }t\ge 0.
\end{align}
The quadratic variation $[X]$ is an increasing, right-continuous and adapted process with $[X](0)=0$; see \cite[Th. 26.5]{met82}. Since the paths are monotone, $[X]$ is of bounded variation on finite intervals. For any $H$-valued semimartingales $X$ and $Y$, we may introduce the notion of covariation via the polarisation identity by setting
\[[ X,Y ]:=\frac{1}{4}\left([X+Y]-[X-Y]\right).\]
Similarly to the real-valued case, it can be shown that the covariation is bilinear. 

The following result on the continuity of the quadratic variation (and covariation) in the metric $\dem$ in the real-valued case can be found in \cite[Prop.\ 2.10]{kardaras2013closure}. A similar argument yields the following extension to the Hilbert space-valued setting:
\begin{lemma}\label{le.cont_of_quad_variation}
	Let  $(X_n)_{n \in \mathbb{N}}$ be a sequence of $H$-valued semimartingales, and let $X$ be an $H$-valued semimartingale. Then, for every $T>0$, the following implication holds:
    \[\lim_{n\rightarrow \infty}\dem_T(X_n,X)= 0 \quad\Rightarrow \quad \lim_{n\rightarrow \infty}\dem_T([X_n],[X])=0. \]
\end{lemma}

\begin{proof} 
Let ${\rm TV}(f)(t)$ denote the total variation of the function $f\colon [0,\infty)\to\R$ on $[0,t]$.
Since $\dem_T([X_n],[X])=\dem_T([X_n]-[X],0)$, and the sequence $([X_n]-[X])_{n\in \mathbb{N}}$ consists of real-valued semimartingales of finite variation, it follows from \cite[Pr.\ 2.7]{kardaras2013closure} that
\begin{equation}\label{eq.Emery_finite_var}
    \lim_{n\rightarrow \infty}E\left[\big(
{\rm TV}\left([X_n]-[X]\right)(T)\big)\wedge 1 \right]=0 \quad \text{implies}\quad \lim_{n\rightarrow \infty}\dem_T([X_n],[X])=0.
\end{equation}
Moreover, as in the real-valued case, it follows from the characterisation of covariation as a limit over partitions (see \cite[Th.\ 26.5]{met82}) and Hölder’s inequality that the real-valued covariation process of two $H$-valued semimartingales $Y$ and $Z$ satisfies $\big([Y,Z](t)\big)^2 \le [Y](t) [Z](t)$ for all $ t \ge 0$. From this, we deduce, that
${\rm TV}[Y,Z](T)\le \sqrt{[Y](T)}\sqrt{[Z](T)}$.
Using this inequality and bilinearity of the covariation, we get for each  $t\in [0,T]$ that
\begin{align*}
       {\rm TV}\left([X_n]-[X]\right)(t)&={\rm TV}\left([X_n-X]+2[X,X_n-X])\right)(t)\\
       &\leq [X_n-X](T)+2 \sqrt{[X](T)}\sqrt{[X_n-X](T)}\,\quad a.s.
\end{align*}
Consequently, it is sufficient by \eqref{eq.Emery_finite_var} to show  that $[X_n-X](T)$ converges to $0$ in probability. 
For this, we may assume $X=0$ and it remains  to show that 
\[\lim_{n\rightarrow \infty} E[[X_n](T)\wedge 1]=0.\]
    
    As $\dem_T(X_n,0)\to 0$ implies $\ducp_T(X_n,0)\to 0$, it follows that $X_n(0)$ and $X_n(T)$ both converge to $0$ in probability. Hence, by Equation \eqref{def.quad_var}, it remains only to show that $(I_n(T))_{n\in\N}$ converges to $0$ in probability, where 
	\[I_n(T):=\int_{(0,T]} \langle X_n(s-),\cdot\rangle\,{\rm d}X_n(s).\]
	By the subsequence principle, it suffices to show that any subsequence has a further subsequence that converges to $0$ in $L_P^0(\Omega;\mathbb{R})$. To this end, let $(I_{n_m}(T))_{m \in \mathbb{N}}$ be a subsequence of $(I_n(T))_{n\in \mathbb{N}}$.  Since our assumptions imply that $\dem_T(X_{n_m},0)\to 0$ as $m\rightarrow \infty$, it follows that $\ducp_T(X_{n_m},0)\to 0$ as $m\rightarrow \infty$. By Lemma~\ref{lem:sum_of_processes} there exists a subsequence  $(X_{n_{m_k}})_{k \in \mathbb{N}}$ such that the sum $\xi(t)=\sum_{k=1}^\infty \norm{X_{n_{m_k}}(t)}$ defines a real-valued c\`adl\`ag process. 
    
   Using $\xi$, we define a sequence $(\tau_\ell)_{\ell \in \mathbb{N}}$ of stopping times by 
    \[\tau_\ell:=\inf\left\{t\in [0,T]:{\xi(t)}>\ell\right\} \quad \text{for all }\ell \in \mathbb{N},\]
    where we note that since $\xi$ is c\`adl\`ag, we have $P(\cup_{\ell=1}^{\infty}\{\tau_\ell=\infty\})=1$.
    Fix $e \in H$ with $\norm{e}=1$, and define for each $k,\ell \in \mathbb{N}$ a process
    \[Y_{k,\ell}(t):=\langle \ell^{-1}X_{n_{m_k}}(t-)\mathbb{1}_{(0,\tau_\ell \wedge T]}(t),\cdot\rangle \,e \qquad \text{for all }t \in [0,T].\]
Clearly, for each $k,\ell \in \mathbb{N}$ we have that $Y_{k,\ell}(t)\in L(H)$ for all $t \in [0,T]$, moreover, it follows from the very definition of $\tau_\ell$ that $\sup_{t \in [0,T]}\norm{Y_{k,\ell}(t)}_{H \rightarrow H}\leq 1$. Approximating $Y_{k,\ell}$ by a sequence $(S_{k,\ell,r})_{r \in \mathbb{N}} \subseteq\mathcal{S}_{{\rm prd}}^{1,{\rm op}}$ and using \cite[Th. 26.3]{met82}, we get
	\begin{align*}
		E[\vert I_{n_{m_k}}(\tau_\ell \wedge T)\vert\wedge 1]&\leq \max\{\ell,1\}\,E\left[\left\vert \int_{(0,T]} \langle \ell^{-1}X_{n_{m_k}}(s-)\mathbb{1}_{(0,\tau_\ell]}(s),\cdot\rangle\,{\rm d}X_{n_{m_k}}(s) \right\vert\wedge 1\right]\\
        &=\max\{\ell,1\}\,E\left[\left(\left\vert \int_{(0,T]} \langle \ell^{-1}X_{n_{m_k}}(s-)\mathbb{1}_{(0,\tau_\ell]}(s),\cdot\rangle\,{\rm d}X_{n_{m_k}}(s) \right\vert \norm{e}\right)\wedge 1\right]\\
        &=\max\{\ell,1\}\,E\left[\norm{ \int_{(0,T]} Y_{k,\ell}(s)\,{\rm d}X_{n_{m_k}}(s) }\wedge 1\right]\\
        &=\max\{\ell,1\}\lim_{r \rightarrow \infty}E\left[\norm{ \int_{(0,T]} S_{k,\ell,r}(s)\,{\rm d}X_{n_{m_k}}(s) }\wedge 1\right]\\
    &\leq \max\{\ell,1\}\sup_{\Gamma \in \mathcal{S}_{{\rm prd}}^{1,{\rm op}}}E\left[\norm{ \int_{(0,T]}  \Gamma(s)\,{\rm d}X_{n_{m_k}}(s) }\wedge 1\right],
	\end{align*}
which, since $\dem_T(X_{n_{m_k}},0)\to 0$ as $k\rightarrow \infty$, implies that
\[\lim_{k \rightarrow \infty}E[\vert I_{n_{m_k}}(\tau_\ell \wedge T)\vert\wedge 1]=0\qquad \text{for all }\ell \in \mathbb{N}.\]
From this, we can deduce that
\begin{align*}
    \lim_{k \rightarrow \infty}E[\vert I_{n_{m_k}}(T)\mathbb{1}_{\{\tau_\ell=\infty\}}\vert\wedge 1]&=\lim_{k \rightarrow \infty}E[\vert I_{n_{m_k}}(\tau_\ell\wedge T)\mathbb{1}_{\{\tau_\ell=\infty\}}\vert\wedge 1]\\
    &\leq \lim_{k \rightarrow \infty}E[\vert I_{n_{m_k}}(\tau_\ell \wedge T)\vert\wedge 1]=0 \qquad \text{for all }\ell \in \mathbb{N}.
\end{align*}
Since $P(\cup_{\ell=1}^{\infty}\{\tau_\ell=\infty\})=1$, or equivalently $P(\cap_{\ell=1}^{\infty}\{\tau_\ell=\infty\}^c)=0$, we have
\[\lim_{\ell\rightarrow \infty}P(\{\tau_\ell=\infty\}^c)=0.\]
Let $\epsilon>0$ be fixed. By using the simple estimate
\[E[\vert I_{n_{m_k}}(T)\vert\wedge 1]\leq E[\vert I_{n_{m_k}}(T)\mathbb{1}_{\{\tau_\ell=\infty\}}\vert\wedge 1] +P(\{\tau_\ell=\infty\}^c),\]
we first choose $\ell^*$ large enough so that $P(\{\tau_{\ell^*}=\infty\}^c)<\epsilon/2$. Then, having fixed $\ell^*$, we now take $k$ to be large enough so that $E[\vert I_{n_{m_k}}(T)\mathbb{1}_{\{\tau_{\ell^*}=\infty\}}\vert\wedge 1]<\epsilon/2$. Since $\epsilon>0$ was arbitrary, we may conclude that
\[\lim_{k \rightarrow \infty}E[\vert I_{n_{m_k}}(T)\vert\wedge 1]=0.\]
Since the argument holds for an arbitrary subsequence $(I_{n_m}(T))_{m \in \mathbb{N}}$, the subsequence principle allows us to conclude that $\lim_{n \rightarrow \infty}E[\vert I_{n}(T)\vert\wedge 1]=0$, which finishes the proof of our claim.
\end{proof}

\section{Cylindrical L\'evy processes and their stochastic integrals}

A cylindrical random variable $Z$ in a Hilbert space $U$ is a linear and continuous mapping $Z\colon U \rightarrow L_{P}^0(\Omega;\mathbb{R})$. 
Its characteristic function is defined as 
\[\varphi_Z\colon U \rightarrow \mathbb{C}, \qquad\varphi_Z(u)=E\big[e^{iZu}\big].\]

Let $\F=(\F_t)_{t\ge 0}$ be a filtration on $(\Omega,\mathcal{A},P)$ satisfying the usual conditions. 
A family $(L(t):t\geq 0)$ of cylindrical random variables $L(t)\colon U \rightarrow L_{P}^0(\Omega;\mathbb{R})$ is called a cylindrical Lévy process if, for every $n \in \mathbb{N}$ and every $u_1,...,u_n \in U$, the  $\mathbb{R}^n$-valued stochastic process
\[\big(\big(L(t)u_1,...,L(t)u_n\big): t \geq 0\big)\]
is a Lévy process with respect to $\F$.  

The characteristic function of $L(t)$ is, for each $t\geq 0$, of the form
\[\varphi_{L(t)}:U\rightarrow \mathbb{C}, \qquad \varphi_{L(t)}(u)=\exp\left(tS(u)\right),\]
where the mapping $S:U \rightarrow \mathbb{C}$  is given by
\[S(u)=ia(u)-\frac{1}{2}\langle Qu,u \rangle+ \int_U\left(e^{i\langle u,h\rangle}-1-i\langle u,h \rangle \mathbb{1}_{B_\mathbb{R}}(\langle u,h \rangle)\right)\, \lambda({\rm d}h).\]
Here,  $a:U\rightarrow \mathbb{R}$ is a continuous mapping with $a(0)=0$, and $Q:U \rightarrow U$ is a positive and symmetric operator. The set function $\lambda\colon \mathcal{Z}_\ast(U)\to [0,\infty]$ is defined on the algebra
\begin{align*}
	\mathcal{Z}_\ast(U)\!:=\big\{ \{u\in U:\! (\scapro{u}{u_1^\ast},\dots, \scapro{u}{u_n^\ast})\in A\}:
	u_1^\ast, \dots, u_n^\ast\in U,  A\in\Borel(\R^n\setminus\{0\}), n\in\N    \big\}, 
\end{align*} 
and satisfies for each $C:=\{u\in U:\! (\scapro{u}{u_1^\ast},\dots, \scapro{u}{u_n^\ast})\in A\}\in \mathcal{Z}_\ast(U)$ that 
\[
\lambda(C)= \lambda_{u_1,\dots, u_n}(A),
\]
where $\lambda_{u_1,\dots, u_n}$ is the L\'evy measure of $\big(\big(L(t)u_1,...,L(t)u_n\big): t \geq 0\big)$. The triplet $[a,Q,\lambda]$ is called the characteristics of $L$. 
See \cite{Applebaum-Riedle} and  \cite{riedle2011infinitely} for this and related results.

Stochastic integration with respect to cylindrical Lévy processes was introduced in the recent work \cite{BR} by two of us. In the present paper, however, we require additional properties of the stochastic integral that, while well established in the case of genuine Lévy processes, must be carefully verified in the cylindrical setting. In many instances, the cylindrical nature of the integrator introduces significant analytical challenges.

We begin with a brief review of the construction of the stochastic integral, referring to \cite{BR} for full details. Fix $T>0$ and let $L$ be a cylindrical Lévy process  in $U$ with characteristics $[a,Q,\lambda]$.  An $L_2(U,H)$-valued predictable step process $\Psi \colon \Omega \times [0,T]\rightarrow L_2(U,H)$ is of the form
\begin{align} \label{eq.step-HS}
	\Psi(\omega, t)=\sum_{i=1}^{n-1} \left(\sum_{k=1}^{N(i)}F_{i,k}\mathbbm{1}_{A_{i,k}}(\omega)\right) \mathbbm{1}_{ (t_i,t_{i+1}]}(t),
\end{align}
where $0=t_1<\cdots < t_{n}=T$, $A_{i,k} \in \mathcal{F}_{t_i}$ and $F_{i,k} \in L_2(U,H)$ for all $i=1,...,n-1$ and $k=1,...,N(i)$. The space of all $L_2(U,H)$-valued predictable step processes is denoted by $\mathcal{S}_{\rm prd}^{\rm HS}:=\mathcal{S}_{\rm prd}^{\rm HS}(U,H)$.

Let $\Psi \in \mathcal{S}_{\rm prd}^{\rm HS}$ be of the form (\ref{eq.step-HS}). It follows from \cite[Th.\ VI.5.2 and Prop.\ VI.5.3]{vakhania_1981} that there exists an $H$-valued random variable $F_{i,k}(L(t_{i+1})-L(t_i))\colon \Omega\to H$ for each $i=1,...,n-1$ and $k=1,...,N(i)$, satisfying \[ \big(L(t_{i+1})-L(t_i)\big)(F_{i,k}^\dagger h)=\langle F_{i,k}(L(t_{i+1})-L(t_i)),h \rangle \quad \text{$P$-a.s.\ for all}\; h \in H.\]
Using these $H$-valued random variables,  the stochastic integral of $\Psi$ is defined, for all $t\in [0,T]$, by
\[I(\Psi)(t):=\int_{(0,t]} \Psi(t) \, {\rm d}L(t) :=\sum_{i=1}^{n-1}\sum_{k=1}^{N(i)} \mathbb{1}_{A_{i,k}}F_{i,k} (L(t_{i+1}\wedge t)-L(t_i\wedge t) ).\]
Thus, the integral $I(\Psi)(t):\Omega\rightarrow H$ is a genuine $H$-valued random variable.

We consider the measure space $\big(\Omega \times [0,T],\mathcal{P}, P_T\big)$, where $\mathcal{P}$ denotes the predictable $\sigma$-algebra, generated by the left-continuous,  $\{\F_t\}_{t\ge 0}$-adapted processes, and the measure $P_T$ is defined by $P_T:=P \otimes {\rm Leb} \vert_{[0,T]}$.

\begin{definition} \label{def.pred_integrability}
	A predictable process $\Psi \colon \Omega \times [0,T] \rightarrow L_2(U,H)$ is $L$-integrable if there exists a sequence $(\Psi_n)_{n \in \mathbb{N}}$ of processes in $\mathcal{S}_{\rm prd}^{\rm HS}$ such that
	\begin{enumerate}[(1)]
		\item[{\rm (1)}]  $(\Psi_n)_{n \in \mathbb{N}}$ converges $P_T$-a.e. to $\Psi$;
		\item[{\rm (2)}] $\displaystyle \lim_{m,n \rightarrow \infty}d_T^{em}(I(\Psi_m),I(\Psi_n))=0.$
	\end{enumerate}
	In this case,  the stochastic integral process $(I(\Psi)(t):\, t\in [0,T])$ of $\Psi$ is defined as the limit 
	of $(I(\Psi_n))_{n\in\N}$ in $\dem_T$. The class of all $L$-integrable $L_2(U,H)$-valued predictable processes will be denoted by $\mathcal{I}_{{\rm prd},L}^{\rm HS}:=\mathcal{I}_{{\rm prd},L}^{\rm HS}(U,H)$. 
\end{definition}

It follows immediately from  the fact that convergence in $\dem_T$ implies convergence in $\ducp_T$, that the integral process has c\`adl\`ag paths. Furthermore,  as shown in \cite{BR}, the integral process is an $H$-valued semimartingale.

To characterise the space of $L$-integrable predictable processes, fix a cylindrical L\'evy process $L$ in $U$ with characteristics $[a,Q,\lambda]$. 
If  $F\colon U\to H$ is a Hilbert-Schmidt operator in $L_2(U,H)$, then there exists a genuine L\'evy process $FL$ in $H$ satisfying 
\[
\scapro{FL(t)}{h}=L(t)(F^\ast h) \qquad\text{for all }t\ge 0, h\in H. 
\]
The L\'evy process $FL$ has genuine characteristics $(a_F, Q_F,\lambda_F)$ satisfying
\begin{align*}
	&\scapro{a_{F}}{h} = a(F^\dagger h)+\int_H \big( \scapro{\theta(u)}{h} - \scapro{u}{h} \1_{[-1,1]}(\scapro{u}{h})\big) \,\lambda_F(\d u),\\
	& Q_F=FQF^\dagger,\qquad 
	\lambda_F=  \lambda\circ F^{-1} \quad\text{on }{\mathcal Z}_\ast(H).
\end{align*}
Here, the characteristics $(a_F, Q_F,\lambda_F)$ is given in terms of the truncation function
\[
\theta\colon H\to H, \qquad \theta(u)=\begin{cases}
u, &\text{if } \norm{u}\le 1, \\
	 \tfrac{u}{\norm{u}}, &\text{otherwise.}
\end{cases}
\] 
Using this characteristics, we define functions 
\begin{align*}
	\zeta_L:L_2(U,H) \rightarrow \R,\qquad 	\zeta_L(F)&= \int_H\Big(\norm{h}^2 \wedge 1\Big)\, \lambda_F(\d h)+{\rm Tr}(FQF^\dagger),\\
	\eta_L:L_2(U,H) \rightarrow \R,\qquad	\eta_L(F)&= \sup_{O  \in E_{L(H)}} \norm{a_{O F}},
\end{align*}
where $E_{L(H)}$ denotes the closed unit ball in $L(H,H)$. 

	For a measurable function $\psi:[0,T]\rightarrow L_2(U,H)$  define
	\begin{align*}
		&m_L(\psi):=\int_{(0,T]} \Big( \zeta_L(\psi(t))+\eta_L(\psi(t))+\big(\norm{\psi(t)}_{2}^2\wedge 1\big) \Big)\,\d t,
	\end{align*}
	and denote by $\M_L$ the space of Lebesgue a.e.\ equivalence classes of measurable functions $\psi:[0,T]\rightarrow L_2(U,H)$ for which $m_L(\psi)< \infty$.

The next result provides a characterisation of the space of predictable processes that are integrable with respect to a given cylindrical Lévy process. The proof, presented in \cite{BR}, is based on a decoupling method introduced in \cite{kwapien_woyczynski_1992}.
\begin{theorem} \label{pred_iff_integrable}
	A predictable Hilbert-Schmidt operator-valued process $\Psi \colon \Omega\times [0,T]\rightarrow L_2(U,H)$ is integrable with respect to $L$ if and only if 
	 $m_L(\Psi)<\infty$ a.s.
\end{theorem}
\begin{proof}
	See Theorem 4.4 in  \cite{BR}. 
\end{proof}

\begin{remark}\label{re.caglad_is_integrable}
    It follows from Theorem \ref{pred_iff_integrable} that every adapted càglàd process $\Psi$ is $L$-integrable. Indeed, since every càglàd function on a compact interval is bounded, it follows from \cite[Re. 3.17]{BR}, that, for each $\omega \in \Omega$, there exists a finite constant $C_\omega$ such that $m_L(\Psi(\omega))\leq C_\omega$. Hence, $\Psi$ is $L$-integrable by Theorem \ref{pred_iff_integrable}.
\end{remark}

The following stochastic dominated convergence result will be of great importance throughout this paper.

	\begin{theorem}\label{th_dom_stoch_conv}
	Let $(\Psi_n)_{n \in \mathbb{N}}$ be a sequence of $L$-integrable, predictable processes 
	converging to an $L_2(U,H)$-valued predictable process $\Psi$  $P_T$-a.e.  If there exists an $L$-integrable predictable  process $\Upsilon $ satisfying for all $n \in \mathbb{N}$ that
	\[(\zeta_L+\eta_L)(\Psi_n(\omega,t)) \leq (\zeta_L+\eta_L)(\Upsilon(\omega,t)) \quad  \text{for $P_T$-a.a.\ $(\omega,t)\in \Omega\times [0,T]$}, \]
	then $\Psi$ is $L$-integrable  and satisfies
	\[\displaystyle \lim_{n\to\infty} d_T^{em}(I(\Psi_n),I(\Psi))=0.\]
\end{theorem}
\begin{proof}
	See Theorem 7.5 in \cite{BR}. 
\end{proof}

\begin{corollary}\label{co.bdd_stoch_convergence}
    Let $(\Psi_n)_{n \in \mathbb{N}}$ be a sequence of $L$-integrable, predictable processes 
	converging to an $L_2(U,H)$-valued predictable process $\Psi$  $P_T$-a.e. If we have
    \[\esssup_{\omega\in\Omega}\sup_{s\in (0,T]}\sup_{n\in\N} \norm{\Psi_n(\omega,s)}<\infty,\]
    then $\Psi$ is $L$-integrable  and satisfies
	\[\displaystyle \lim_{n\to\infty} d_T^{em}(I(\Psi_n),I(\Psi))=0.\]
\end{corollary}
\begin{proof}
    The proof is completely analogous to that of \cite[Th. 7.5]{BR}, with the exception that we can now use \cite[Re. 3.17]{BR} to obtain an upper bound on $m_L$ by a constant.
\end{proof}

 The following lemma provides a tool to obtain upper bounds for $\zeta_L$ and $\eta_L$.
\begin{lemma}\label{le.modular_domination}
    Let $U,H$ and $V$ be separable Hilbert spaces, and $L$ be a cylindrical L\'evy process in $U$. For each $C \in L_2(U,H)$ and $D \in L(H,V)$ with $\norm{D}\leq r$ for some $r>0$, the following assertions hold:
    \begin{enumerate}
        \item [{\rm (1)}] $\zeta_L(DC)\leq \max\{r,1\}\, \zeta_L(C)$;
        \item [{\rm (2)}] $\eta_L(DC)\leq d_r\left(\eta_L(C)+ \zeta_L(C)\right)$ for some constant $d_r>0$ depending only on $r$.
    \end{enumerate}
\end{lemma}
\begin{proof}
The proof of (1) is completely analogous to the second part of \cite[Le.\ 3.14]{BR}, while (2) follows from \cite[Le.\ 3.9]{BR}, \cite[Re.\ 3.10]{BR}, and an argument similar to \cite[Le.\ 3.16]{BR}.
\end{proof}

Throughout this paper, we will use the following properties of stochastic integrals. Although these results are standard in classical stochastic analysis, see e.g.\ \cite[Th.\ 2.7]{Chung-Williams}, we provide detailed proofs due to the cylindrical nature of our integrator. 
\begin{lemma}\label{le.prop_int_proc}
Let $L$ be a cylindrical L\'evy process in $U$, and let $\Psi$ be a predictable, $L_2(U,H)$-valued process that is integrable with respect to $L$. Then the integral process $(Y(t):\, t\in [0,T])$, defined by $Y(t):=\int_{(0,t]} \Psi\,{\rm d}L$, satisfies the following properties: 
	\begin{enumerate}
        \item [{\rm (a)}] Let $V$  be a separable Hilbert space  and $B \in L(H,V)$. Then, for all $0\leq p<q\leq T$ and $S \in \mathcal{F}_{p}$, we have that $\mathbb{1}_S B\Psi$ is $L$-integrable and
        	\begin{equation}\label{eq.op-commutes-int}
        		\mathbb{1}_{S}B \left(\int_{(p,q]} \Psi(s)\,{\rm d}L(s)\right)=\int_{(p,q]} \mathbb{1}_S B\Psi(s)\,{\rm d}L(s).
        \end{equation}
		\item [{\rm (b)}] Let $\Phi:\Omega\times[0,T]\rightarrow H$ be a uniformly bounded, adapted, càglàd process. Then, for each fixed $t>0$,  it follows that $\langle \Psi^\dagger \Phi, \cdot \rangle$ is $L$-integrable and
		\[\int_{(0,t]} \langle\Phi(s),\cdot\rangle \, {\rm d}Y(s)=\int_{(0,t]} \big\langle \Psi^\dagger(s)\Phi(s),\cdot\big\rangle \, {\rm d}L(s).\]
		\item [{\rm (c)}] For every stopping time $\tau$, we have that $\1_{(0,\tau]}\Psi$ is $L$-integrable and
		\[Y(\tau \wedge T)=\int_{(0,T]} \1_{(0,\tau]}(s)\Psi(s)\,{\rm d}L(s).\]
	\end{enumerate}
\end{lemma}

\begin{proof}
(a): First, observe that identity \eqref{eq.op-commutes-int}  holds for all $\Psi \in \mathcal{S}_{\rm prd}^{\rm HS}$ by direct computation using the definition of the stochastic integral.

Now consider a general integrand  $\Psi \in \mathcal{I}_{{\rm prd},L}^{\rm HS}(U,H)$. By definition, there exists a sequence $(\Psi_n)_{n\in \mathbb{N}}\subseteq \mathcal{S}_{\rm prd}^{\rm HS}$ converging $P_T$-a.e.\ to $\Psi$ and 
such that $d_T^{em}(I(\Psi_m),I(\Psi_n))\to 0$ as $m,n\to\infty$. Since  identity  \eqref{eq.op-commutes-int} applies to each $\Psi_n$,  Lemma \ref{le.cont_of_semimartingales} implies that $d_T^{em}(I(\1_SB\Psi_m),I(\1_S B\Psi_n))\to 0$. Since $(\1_SB\Psi_n)_{n\in\N}$ converges to $\1_S B\Psi$ $P_T$-a.e., 
it follows that ${\mathbb{1}_SB\Psi}$ is $L$-integrable, and 
\[\lim_{n\rightarrow \infty}\int_{(p,q]} \mathbb{1}_S B\Psi_n(s)\,{\rm d}L(s)=\int_{(p,q]} \mathbb{1}_S B\Psi(s)\,{\rm d}L(s)\quad \text{in } L_P^0(\Omega,V),\]
On the other hand, by the continuity of $B$, we have 
\[\lim_{n\rightarrow \infty}\mathbb{1}_{S}B \left(\int_{(p,q]} \Psi_n(s)\,{\rm d}L(s)\right)=\mathbb{1}_{S}B \left(\int_{(p,q]} \Psi(s)\,{\rm d}L(s)\right)\quad \text{in } L_P^0(\Omega,V). \]
Since equality \eqref{eq.op-commutes-int} holds for each $\Psi_n$, the proof is completed.

(b):  We first establish that the process $\langle \Psi^\dagger\Phi,\cdot\rangle$ is $L$-integrable. Since $\Phi$ is uniformly bounded and the space of $L$-integrable predictable processes is a vector space, dividing by the uniform bound of $\Phi$ if necessary, we may assume that $\Phi$ is uniformly bounded by $1$. By noting that $\langle \Psi(s)^\dagger\Phi(s),\cdot\rangle= \langle \Phi(s) ,\cdot\rangle \circ \Psi(s)$ for all $s\in [0,T]$, we may use Lemma \ref{le.modular_domination} to establish that there exists a constant $c>0$ such that $m_L(\langle \Psi^\dagger\Phi,\cdot\rangle)\leq c\, m_L(\Psi)<\infty$ almost surely. The fact that $\langle \Psi^\dagger\Phi,\cdot\rangle$ is $L$-integrable now follows from Theorem \ref{pred_iff_integrable}.

We first show the identity under the assumption that $\Phi$ is of the form
\begin{equation}\label{eq.step-H}\Phi(\omega, t)=\sum_{i=1}^{n-1} \left(\sum_{k=1}^{N(i)}h_{i,k}\mathbbm{1}_{A_{i,k}}(\omega)\right) \mathbbm{1}_{ (t_i,t_{i+1}]}(t),
\end{equation} 
where $0=t_1<\cdots < t_{n}=t$, $A_{i,k} \in \mathcal{F}_{t_i}$ and $h_{i,k} \in H$ for all $i=1,...,n-1$ and $k=1,...,N(i)$. By Part (a) of this lemma,  we obtain
\begin{align*}
	\int_{(0,t]} \langle\Phi(s),\cdot\rangle\, {\rm d}Y(s)&=\sum_{i=1}^{n-1}\left\langle\sum_{k=1}^{N(i)}h_{i,k}\mathbbm{1}_{A_{i,k}},Y(t_{i+1})-Y(t_i)\right \rangle\\
	&=\sum_{i=1}^{n-1} \int_{(t_i,t_{i+1}]}\left\langle\Psi^\dagger(s)\Bigg(\sum_{k=1}^{N(i)}h_{i,k}\mathbbm{1}_{A_{i,k}}\Bigg),\cdot\right \rangle\,{\rm d}L(s)\\
	&=\int_{(0,t]} \langle\Psi^\dagger(s)\Phi(s),\cdot\rangle \, {\rm d}L(s).
\end{align*}

Now consider an arbitrary uniformly bounded, adapted, càglàd process $\Phi \colon \Omega \times [0,T] \to H$. By \cite[Lem.~1.2.19 and Rem.~1.2.20]{HNVW}, there exists a sequence $(\Phi_n)_{n \in \mathbb{N}}$ of processes of the form \eqref{eq.step-H} such that $\Phi_n \to \Phi$ $P_T$-a.e. and satisfying 
\begin{equation}\label{eq.bound on Phi_n}  c:=\esssup_{\omega\in\Omega}\sup_{s\in (0,T]}\sup_{n\in\N} \norm{\Phi_n(\omega,s)}<\infty. 
\end{equation}
 By Theorem~26.3 in \cite{met82}, for each fixed $t > 0$, we obtain
\begin{equation}\label{eq.n_limit_1}
	\lim_{n\rightarrow \infty}\int_{(0,t]} \langle \Phi_n(s),\cdot \rangle\, {\rm d}Y(s)=
	\int_{(0,t]} \langle \Phi(s),\cdot \rangle\, {\rm d}Y(s)
	\qquad\text{in } L_P^0(\Omega;\R).
\end{equation}
In order to complete the proof, it remains to establish that, for each $t>0$, 
\begin{equation}\label{eq.n_limit_2}
	\lim_{n\to\infty}\int_{(0,t]} \left\langle \Psi^\dagger(s)\Phi_n(s), \cdot \right\rangle\,{\rm d}L(s)
	= \int_{(0,t]} \left\langle \Psi^\dagger(s)\Phi(s), \cdot \right\rangle\,{\rm d}L(s)
	\qquad \text{in } L_P^0(\Omega;\R).
\end{equation}
According to Corollary~7.3 in~\cite{BR}, this is the case if and only if 
\begin{equation}\label{eq.modular_convergence}
	\lim_{n\to\infty} 
	E\!\left[\, m_L\!\left(\scapro{\Psi^\dagger(\Phi_n-\Phi)}{\cdot}\right)\wedge 1 \right]
	= 0.
\end{equation}
As a first step, we observe that, since $\Psi$ is $L$-integrable, there exists a set $\Omega_\Psi \subseteq \Omega$ with $P(\Omega_\Psi)=1$ such that, for every $\omega \in \Omega_\Psi$, 
\[m_L(\Psi(\omega))=\int_{(0,T]} \Big( \zeta_L(\Psi(\omega,t))+\eta_L(\Psi(\omega,t))+\big(\norm{\Psi(\omega,t)}_{2}^2\wedge 1\big) \Big)\,\d t<\infty.\]
By Lemma~\ref{le.modular_domination}, it follows that, for all $\omega \in \Omega_\Psi$ and $t \in (0,T]$, 
\begin{align*}
    \zeta_L\left(\scapro{\Psi^\dagger(\omega,t)(\Phi_n-\Phi)(\omega,t)}{\cdot}\right)&\leq \max\{2c,1\}\, \zeta_L(\Psi(\omega,t));\\
   \eta_L\left(\scapro{\Psi^\dagger(\omega,t)(\Phi_n-\Phi)(\omega,t)}{\cdot}\right)&\leq d_{2c}\left(\eta_L(\Psi(\omega,t))+\zeta_L(\Psi(\omega,t))\right).
\end{align*}
Moreover, a straightforward calculation yields, for all $\omega \in \Omega_\Psi$ and $t \in (0,T]$, 
\[
\|\langle \Psi^\dagger(\omega,t)(\Phi_n-\Phi)(\omega,t), \cdot \rangle\|_{2}^2
= \|\langle (\Phi_n-\Phi)(\omega,t), \cdot \rangle \circ \Psi(\omega,t)\|_{2}^2
\le 2\,\|\Psi(\omega,t)\|_{2}^2.
\]
Using the above bounds, we may apply Lebesgue’s dominated convergence theorem for each fixed $\omega \in \Omega_\Psi$, with the dominating function
\[
t \mapsto 3\,\zeta_L(\Psi(\omega,t))
+ \eta_L(\Psi(\omega,t))
+ 2\big(\|\Psi(\omega,t)\|_{2}^2 \wedge 1\big),
\]
to obtain
\[
\lim_{n\to\infty}
m_L \big(\langle \Psi^\dagger(\omega)(\Phi_n-\Phi)(\omega), \cdot \rangle\big)
= 0
\qquad \text{for all } \omega \in \Omega_\Psi.
\]
Since $P(\Omega_\Psi)=1$, this proves \eqref{eq.modular_convergence} and thus completes the proof of~(b).

(c): Let $\sigma:=\tau \wedge T$ and define $\sigma_n=2^{-n}[2^n\sigma+1]$ and $N_n=[2^n\,T]$ for each $n \in \mathbb{N}$, where $[x]$ denotes the integer part of any real $x$. Define the following subset of $\Omega\times [0,\infty)$: 
	\begin{equation*}
		J_n:=\bigcup_{k=0}^{N_n} \{\sigma\geq k2^{-n}\}\times (k2^{-n},(k+1)2^{-n}].
	\end{equation*}
	For each $n \in \mathbb{N}$, we have
    \begin{align*}
        Y(\sigma_n)&=\sum_{k=0}^{N_n} \mathbb{1}_{\left\{{k2^{-n} \leq \sigma < (k+1)2^{-n}}\right\}}Y((k+1)2^{-n})\nonumber\\
        &=\sum_{k=0}^{N_n} \left( \mathbb{1}_{\left\{{k2^{-n} \leq \sigma }\right\}}- \mathbb{1}_{\left\{ {(k+1)2^{-n} }\le \sigma\right\}}\right) Y((k+1)2^{-n})\nonumber\\
		&=\sum_{k=0}^{N_n} \mathbb{1}_{\{k2^{-n}\le \sigma\}}\Big(Y((k+1)2^{-n})-Y(k2^{-n})\Big).
    \end{align*}
Since  the filtration $\F$ is right-continuous, the set $\{ k2^{-n}\le \sigma\}$ is $\mathcal{F}_{k2^{-n}}$-measurable. Hence, by  Part (a) of this lemma and linearity of the integral, we obtain
    \begin{align*}
        Y(\sigma_n)&= \sum_{k=0}^{N_n} \mathbb{1}_{\{\sigma \geq k2^{-n}\}}\int_{(k2^{-n},(k+1)2^{-n}]} \Psi \,{\rm d}L\nonumber\\
        &=\sum_{k=0}^{N_n} \int_{(0,T]} \mathbb{1}_{(k2^{-n},(k+1)2^{-n}]\times \{\sigma\geq k2^{-n}\}}\Psi \, {\rm d}L
        =\int_{(0,T]} \mathbb{1}_{J_n}\Psi\, {\rm d}L. 
    \end{align*}
Since $Y$ has right continuous paths and    $\sigma_n \downarrow \sigma= \tau \wedge T$, we have 
	\[Y(\tau \wedge T) =\lim_{n \rightarrow \infty}Y(\sigma_n) .\]
On the other hand, applying Theorem \ref{th_dom_stoch_conv} with $\Psi$ as an integrable dominating function shows 
	\[\lim_{n\rightarrow \infty}\int_{(0,T]} \mathbb{1}_{J_n}\Psi\, {\rm d}L = \int_{(0,T]} \mathbb{1}_{(0,\tau \wedge T]}\Psi\, {\rm d}L,\]
which completes the proof.
\end{proof}

\begin{lemma}\label{le.quadratic_variation_of_integral}
	Let $L$ be a cylindrical L\'evy process in $U$, and let $\Psi$ be a predictable, $L_2(U,H)$-valued process that is integrable with respect to $L$. Define the integral process $(Y(t):\, t\in [0,T])$  by $Y(t):=\int_{(0,t]} \Psi\,{\rm d}L$. Then, for every uniformly bounded, real-valued, non-negative, adapted c\`agl\`ad process $Z$, the following identity holds:
	\begin{equation}\label{eq.quad_var_int_proc}
		\int_{(0,t]} Z(s)^2\;{\rm d}[Y](s)=\left[\int_{(0,{\cdot}]} Z(s)\Psi(s)\,{\rm d}L(s)\right](t)
		\qquad\text{for all }t\in [0,T].
	\end{equation}
\end{lemma}
\begin{proof} Fix $t\ge 0$, and assume first that $Z$ is of the form
	\begin{align} \label{eq.step-simple-H}
		Z(s)=\sum_{i=1}^{m-1} Z_i \mathbbm{1}_{ (t_i,t_{i+1}]}(s) \qquad\text{for }s\in [0,t], 
	\end{align}
	where $0=t_1<\cdots < t_{m}=t$, and each $Z_i$ is a real-valued, $\mathcal{F}_{t_i}$-measurable random variable. Let  $(\pi_n)_{n \in \mathbb{N}}$ be a sequence of nested partitions of $[0,t]$ with  mesh tending to $0$ and such that $\{t_1,\dots, t_n\}\subseteq \pi_n$ for all $n\in\N$. Corollary 26.7.2 in \cite{met82} yields
		\begin{align*}    
		\int_{(0,t]} Z(s)^2 {\rm d}[Y](s)
		&= \lim_{n \rightarrow \infty} \sum_{k\in \pi_n} Z(t_{k,n})^2 \norm{Y(t_{k+1,n})-Y(t_{k,n})}^2
		\qquad \text{in } L_P^0(\Omega;\R). 
	\end{align*}
    Since $Z$ is non-negative and constant on each interval $(t_{k,n},t_{k+1,n}]$, it follows that 
    \[
     Z(t_{k,n})^2 \norm{Y(t_{k+1,n})-Y(t_{k,n})}^2
     = \norm{Z(t_{k,n}) \int_{(t_{k,n},t_{k+1,n}]} \Psi\,{\rm d}L}^2
     = \norm{ \int_{(t_{k,n},t_{k+1,n}]} Z \Psi\,{\rm d}L}^2.
    \]
   Hence, by Theorem~26.5 in~\cite{met82}, we obtain in $L_P^0(\Omega;\R)$ that
	\begin{align*}    
		\int_{(0,t]} Z(s)^2 {\rm d}[Y](s)
		&= \lim_{n \rightarrow \infty} \sum_{k\in \pi_n}  \norm{ \int_{(t_{k,n},t_{k+1,n}]} Z \Psi\,{\rm d}L}^2
		=
		\left[ 
		\int_{(0,\cdot]} Z(s)\Psi(s) {\rm d}L(s)
		\right](t),
	\end{align*}
which proves Equation \eqref{eq.quad_var_int_proc} for $Z$ of the form \eqref{eq.step-simple-H}.

Now let $Z$ be an arbitrary uniformly bounded, real-valued, non-negative, adapted c\`agl\`ad process. 
Since both sides of~\eqref{eq.quad_var_int_proc} are quadratically homogeneous, we may, without loss of generality, assume that $\lvert Z\rvert \le 1$ $P$-a.s.

By \cite[Rem. 3.6]{BR}, we have $m_L(Z\Psi)\leq m_L(\Psi)<\infty$, which guarantees  that $Z\Psi$ is $L$-integrable. Let $(Z_n)_{n \in \mathbb{N}}$ be a sequence of processes of the form \eqref{eq.step-simple-H} converging $P_T$-a.e.\ to $Z$. Since $Z\le 1$, we may choose the sequence $(Z_n)_{n \in \mathbb{N}}$ so that $\lvert Z_n\rvert \le 1$ $P$-a.s.\ for all $n\in\N$. 
Then, again by~\cite[Rem.~3.6]{BR}, it follows that $m_L(Z_n\Psi)\le m_L(\Psi)$ $P$-a.s., implying that each $Z_n\Psi$ is $L$-integrable. 
Applying Theorem~\ref{th_dom_stoch_conv}, we obtain
	\begin{equation*}
		\lim_{n \rightarrow \infty}\dem_T\left(\int_{(0,\cdot]} Z_n(s)\Psi(s) {\rm d}L(s), \int_{(0,\cdot]} Z(s)\Psi(s) {\rm d}L(s) \right)=0.
	\end{equation*}
By Lemma~\ref{le.cont_of_quad_variation}, this implies
	\begin{equation}\label{eq.step3_1}
		\lim_{n \rightarrow \infty}\dem_T\left(\left[\int_{(0,\cdot]} Z_n(s)\Psi(s) {\rm d}L(s)\right],\left[ \int_{(0,\cdot]} Z(s)\Psi(s) {\rm d}L(s)\right]\right)=0.
	\end{equation}
	On the other hand, since  $\sup_{n \in \mathbb{N}}\abs{Z_n}\leq 1$, Lebesgue's  dominated convergence theorem yields,  for each $t \geq 0$, 
	\begin{equation}\label{eq.step3_2}
		\lim_{n \rightarrow \infty}\int_{(0,t]} Z_n(s)^2 {\rm d}[Y](s)=\int_{(0,t]} Z(s)^2 {\rm d}[Y](s) \quad \text{for $P$-a.a.\ }\omega\in \Omega.
	\end{equation}
	Combining Equations \eqref{eq.step3_1} and \eqref{eq.step3_2} using the initial step, the result follows.
\end{proof}

\section{Existence and uniqueness of solution}

We consider an abstract evolution equation of the form 
\begin{align}\label{eq.equation}
	\begin{split}
\d X(t)&= \big(A X(t)+ \drift(X(t))\big) \d t +  \diff(X(t-))\, \d L(t), \qquad t\ge 0,\\
  X(0)&= X_0, 
\end{split}
\end{align}
where $A$ is the generator of a strongly continuous semigroup $(S(t))_{t\ge 0}$ on $H$, and  
 $\drift\colon H\to H$,  $\diff\colon H\to L_2(U,H)$ are given maps. The initial condition is specified  by an $H$-valued, $\F_0$-measurable random variable $X_0$, and the noise is modelled by an arbitrary cylindrical L\'evy process $L$ on $U$.  
\begin{definition}
A process \(X\) is called a \emph{mild solution} to \eqref{eq.equation} if it is an \(H\)-valued, adapted, c\`adl\`ag process such that, for every \(t \ge 0\),
\[
X(t)=S(t)X_0 + \int_{(0,t]} S(t-s)\drift(X(s))\, \d s + \int_{(0,t]} S(t-s)\diff(X(s-))\, \d L(s) \quad P\text{-a.s.}
\] 
A mild solution is called unique if any two mild solutions are indistinguishable.
\end{definition}
The main finding of this work is the following existence and uniqueness result:
\begin{theorem}\label{th.existence}
Assume that the following assumptions are satisfied:
\begin{enumerate}
	\item [{\rm (A1)}] $(S(t))_{t\ge 0}$  is a contraction semigroup on $H$;
	\item [{\rm (A2)}] there exists a constant $c_\drift$ such that 
	\[   \norm{\drift(h_1)-\drift(h_2)}\le c_\drift \norm{h_1 -h_2}\qquad\text{for all }h_1,h_2\in H; \]
	\item[{\rm (A3)}] there exists a constant $c_\diff$ such that 
	\[   \norm{\diff(h_1)-\diff(h_2)}_{L_2(U,H)}\le c_\diff \norm{h_1 -h_2}\qquad\text{for all }h_1,h_2\in H. \]
\end{enumerate}
Then there exists a unique mild solution to Equation \eqref{eq.equation}.
\end{theorem}
The proof of this theorem is provided at the end of this section. One of the main ingredients of the proof is a stochastic Grönwall result, which we will prove first. To that end, we establish the following inequality:
\begin{lemma}\label{le.pos_semimartingale}
    If $X$ is a real-valued, strictly positive semimartingale then
    \[\ln\left(\frac{X^*(t)}{X(0)}\right)\leq \left(\int_{(0,\cdot ]} \frac{1}{X^*(s-)}\;{\rm d}X(s)\right)^*(t) \quad \text{a.s.\ for all $t\ge 0$.}\]
\end{lemma}

\begin{proof} (The proof originates from \cite{Lowther}). 
Applying It{\^o}'s formula, \cite[Th.\ II.32] {Protter}, to the non-decreasing semimartingale $X^\ast$ and the function $\beta\mapsto \ln \beta$, we obtain, for each $t>0$, that
\begin{align*}
    \ln\left(\frac{X^*(t)}{X(0)}\right)
    &= \int_{(0,t]} \frac{1}{X^\ast(s-)} \, \d X^\ast (s) + \sum_{s\in (0,t]} \left( \ln X^\ast (s) - \ln X^\ast (s-) - \frac{1}{X^\ast (s-)}\Delta X^\ast (s)  \right). 
\end{align*}
Since it follows from the inequality $\ln (1+\beta)\le \beta$ for all $\beta>-1$  that the summation is non-positive, we conclude, for each $t>0$, that
\begin{align}\label{eq.ln_domination}
    \ln\left(\frac{X^*(t)}{X(0)}\right)
    &\le \int_{(0,t]} \frac{1}{X^*(s-)}\;{\rm d}X^*(s).
\end{align}
Since $X^*$ is nondecreasing, the process $1/X^*$ is nonincreasing, and hence also of finite variation. Define
\[
Y(t):=\frac{X(t)}{X^*(t)}, \qquad t\ge0,
\]
so that $X=Y\,X^*$. Applying the integration-by-parts formula for products of semimartingales
(see, e.g., \cite[Cor.~2 to Th.~II.22]{Protter}) yields
\[
\d X(s)=\d\bigl(Y(s)X^*(s)\bigr)
=
Y(s-)\,\d X^*(s)+X^*(s-)\,\d Y(s)+\d[Y,X^*](s).
\]
Dividing by $X^*(s-)$ and integrating over $(0,t]$ gives
\begin{align}\label{eq.decomp}
	\begin{split}
 &	\int_{(0,t]} \frac{1}{X^*(s-)}\,\d X(s)\\
&\qquad 	=
	\int_{(0,t]} \frac{Y(s-)}{X^*(s-)}\,\d X^*(s)
	+\bigl(Y(t)-Y(0)\bigr)
	+\int_{(0,t]} \frac{1}{X^*(s-)}\,\d[Y,X^*](s).
	\end{split}
\end{align}
Since $X^*$ has finite variation, its quadratic covariation with any semimartingale
is purely discontinuous. Hence, by \cite[Th.~II.28]{Protter},
\[
[Y,X^*](t)=Y(0)X^*(0)+\sum_{0<s\le t}\Delta Y(s)\,\Delta X^*(s),
\]
and therefore, for $0<s\le t$,
\[
\Delta[Y,X^*](s)=\Delta Y(s)\,\Delta X^*(s).
\]
Consequently, the stochastic integral with respect to the finite-variation process
$[Y,X^*]$ reduces to a sum over jumps:
\begin{equation}\label{eq.bracket_integral_sum}
	\int_{(0,t]} \frac{1}{X^*(s-)}\,\d [Y,X^*](s)
	=
	\sum_{0<s\le t}\frac{1}{X^*(s-)}\,\Delta Y(s)\,\Delta X^*(s).
\end{equation}
Observe that $\Delta X^*(s)\neq0$ only if $X$ attains a new maximum by a jump at time $s$.
In this case,
\[
X^*(s)=X(s),\qquad X^*(s-)=\sup_{r<s}X(r)\ge X(s-),
\]
so that
\[
Y(s)=1,\qquad Y(s-)=\frac{X(s-)}{X^*(s-)}\le1,
\qquad \Delta Y(s)=1-Y(s-).
\]
If $\Delta X^*(s)=0$, the corresponding summand in \eqref{eq.bracket_integral_sum} vanishes.
Hence,
\begin{equation}\label{eq.bracket_as_compensator}
	\int_{(0,t]} \frac{1}{X^*(s-)}\,\d [Y,X^*](s)
	=
	\sum_{\substack{0<s\le t\\ \Delta X^*(s)>0}}
	\frac{1-Y(s-)}{X^*(s-)}\,\Delta X^*(s).
\end{equation}
Decompose $X^*$ into its continuous and jump parts,
\[
X^*=(X^*)^c+\sum_{0<s\le\cdot}\Delta X^*(s).
\]
On the support of $\d(X^*)^c$, the running maximum increases continuously, and hence
$X(s)=X^*(s)$ at those times, implying $Y(s-)=1$. Therefore,
\[
\int_{(0,t]} \frac{Y(s-)}{X^*(s-)}\,\d (X^*)^c(s)
=
\int_{(0,t]} \frac{1}{X^*(s-)}\,d(X^*)^c(s).
\]
For the jump part we have
\[
\int_{(0,t]} \frac{Y(s-)}{X^*(s-)}\,\ d\Big(\sum_{0<u\le s}\Delta X^*(u)\Big)
=
\sum_{0<s\le t}\frac{Y(s-)}{X^*(s-)}\,\Delta X^*(s).
\]
Combining these identities with \eqref{eq.bracket_as_compensator} yields
\[
\int_{(0,t]} \frac{Y(s-)}{X^*(s-)}\,\d X^*(s)
+\int_{(0,t]} \frac{1}{X^*(s-)}\,\d [Y,X^*](s)
=
\int_{(0,t]} \frac{1}{X^*(s-)}\,\d X^*(s).
\]
Substituting this equality into \eqref{eq.decomp} gives
\begin{equation}\label{eq.change-final}
	\int_{(0,t]} \frac{1}{X^*(s-)}\,\d X(s)
	=
	\frac{X(t)}{X^*(t)}-1
	+\int_{(0,t]} \frac{1}{X^*(s-)}\,\d X^*(s),
	\qquad \text{$P$-a.s.\ for all } t\ge0. 
\end{equation}	
Fix \(t\ge0\), and work on a sample path outside a \(P\)-null set on which \eqref{eq.change-final} holds for all times. Define
\[
\tau:=\inf\{s\in[0,t]:X^*(s)=X^*(t)\}.
\]
Since \(X^*\) is nondecreasing and c\`adl\`ag, we have $\tau \leq t$, and \(X^*\) is constant on \([\tau,t]\).
By definition of \(\tau\), either \(X(\tau)=X^*(\tau)\) or \(X(\tau-)=X^*(\tau-)=X^*(t)\).

If \(X(\tau)=X^*(\tau)\), then applying \eqref{eq.change-final} at time \(\tau\) yields
\[
\int_{(0,\tau]} \frac{1}{X^*(s-)}\,{\rm d}X(s)
=
\int_{(0,\tau]} \frac{1}{X^*(s-)}\,{\rm d}X^*(s).
\]
Since \(X^*\) is constant on \([\tau,t]\), it follows that
\[
\int_{(0,\tau]} \frac{1}{X^*(s-)}\,{\rm d}X^*(s)
=
\int_{(0,t]} \frac{1}{X^*(s-)}\,{\rm d}X^*(s).
\]
If \(X(\tau-)=X^*(\tau-)=X^*(t)\), then, letting \(t\uparrow\tau\) in \eqref{eq.change-final} gives
\[
\int_{(0,\tau)} \frac{1}{X^*(s-)}\,{\rm d}X(s)
=
\int_{(0,\tau)} \frac{1}{X^*(s-)}\,{\rm d}X^*(s).
\]
Since \(X^*\) is constant on \([\tau,t]\) and  \(\Delta X^*(\tau)=0\), we obtain
\[
\int_{(0,\tau)} \frac{1}{X^*(s-)}\,{\rm d}X^*(s)
=
\int_{(0,t]} \frac{1}{X^*(s-)}\,{\rm d}X^*(s).
\]
Therefore, in either case,
\[
\left(\int_{(0,\cdot]} \frac{1}{X^*(s-)}\,{\rm d}X(s)\right)^*(t)
\ge
\int_{(0,t]} \frac{1}{X^*(s-)}\,{\rm d}X^*(s).
\]
Combining this with \eqref{eq.ln_domination} completes the proof.
\end{proof}

Although the next result does not provide a quantitative bound in terms of an exponential function, we refer to it as a stochastic Grönwall result, as it captures the typical application of the Grönwall inequality.
\begin{proposition}\label{pro.ucp_convergence}
    Let $H,U$ and $V$ be separable Hilbert spaces, and $(Y_n)_{n \in \mathbb{N}}$ be a sequence of $H$-valued processes of the form
    \[Y_n(t)={S_n(t)}\left(\int_{(0,t]} \drift_n(s)\,{\rm d}s+\int_{(0,t]} \diff_n(s)\,{\rm d}L(s)\right)+C_n(t)
    \qquad\text{for all }t\ge 0, \]
where, for some constants $K_1,K_2>0$,  the following conditions hold: 
\begin{enumerate}
    {\item [{\rm (1)}] the sequence $(S_n)_{n\in \mathbb{N}}$ consists of $\mathcal{L}(V,H)$-valued, strongly continuous, deterministic functions satisfying $\sup_{n \in \mathbb{N}}\sup_{t\geq 0}\norm{S_n(t)}_{V\rightarrow H}\leq 1$;}
    \item [{\rm (2)}] the sequence $(\drift_n)_{n \in \mathbb{N}}$ consists of {$V$-valued}, $P$-a.s.\ Bochner integrable, predictable processes satisfying, for all $n \in \mathbb{N}$ and $t\geq 0$,  that $\drift_n^*(t)\leq K_1 Y_n^*(t-)$;
    \item [{\rm (3)}] the sequence $(\diff_n)_{n \in \mathbb{N}}$ consists of  {$L_2(U,V)$-valued},  $L$-integrable, predictable processes satisfying, for all $n \in \mathbb{N}$ and $t\geq 0$,  that $\diff_n^*(t)\leq K_2 Y_n^*(t-)$;    \item  [{\rm (4)}] the sequence   $(C_n)_{n \in \mathbb{N}}$ consists of   $H$-valued, càdlàg, adapted processes converging to $0$ uniformly on compacts in probability.
\end{enumerate}
Then  $(Y_n)_{n \in \mathbb{N}}$ converges to $0$ uniformly on compacts in probability.
\end{proposition}

\begin{proof} (The proof originates from \cite{Lowther}). 
We first prove the claim under the additional assumption that there exists a sequence $(\epsilon_n)_{n\in\N}$ of  strictly positive reals converging to $0$ and satisfying 
\begin{align}\label{eq.first-assumption-C_n}
	\sup_{t\geq 0}\norm{C_n(t)}\leq \epsilon_n \qquad\text{a.s.\ for all }n\in\N.
\end{align}
{To this end, we define a sequence $(Z_n)_{n \in \mathbb{N}}$ of $V$-valued processes, and another sequence $(D_n)_{n \in \mathbb{N}}$ of $H$-valued processes for all $t\geq0$ by
\begin{align*}
Z_n(t)&\vcentcolon= \int_{(0,t]} \drift_n(s)\,{\rm d}s+\int_{(0,t]} \diff_n(s)\,{\rm d}L(s);\\
D_n(t)&\vcentcolon=Y_n(t)-C_n(t)=S_n(t)Z_n(t).
\end{align*}
}Since $(Z_n)_{n \in \mathbb{N}}$ is a sequence of $V$-valued semimartingales, it follows from \cite[Co.\ 26.6]{met82} that  the expression $\norm{Z_n}^2+\epsilon_n^2$ defines a positive semimartingale for each $n \in \mathbb{N}$. {Using the elementary fact that the function $\ln$ is increasing,} and applying Lemma \ref{le.pos_semimartingale} to the strictly positive semimartingale $\norm{Z_n}^2+\epsilon_n^2$ demonstrates, for each $t>0$, that 
\begin{align}\label{eq.log_bound_on_ito_process}
    &{\ln\left(\epsilon_n^{-2}\left(D_n^*(t)\right)^2+1\right)}\nonumber\\
    &\quad\leq\ln\left(\epsilon_n^{-2}\left(Z_n^*(t)\right)^2+1\right)\leq \left(\int_{(0,\cdot]} \frac{1}{\left(Z_n^*(s-)\right)^2+\epsilon_n^2}\;{\rm d}{\norm{Z_n}^2}(s)\right)^*(t)\quad\text{ a.s.}
\end{align}
Define the predictable processes $\alpha_n$, $\beta_n$  and $\beta_n^\dagger$ by
\[
\alpha_n(s): =\frac{\drift_n(s)}{(Z_n^*(s-)+\epsilon_n)}, 
\qquad 
\beta_n(s) : = \frac{\diff_n(s)}{(Z_n^*(s-)+\epsilon_n)}, 
\quad 
\beta_n^\dagger(s) : = \frac{\diff_n^\dagger(s)}{(Z_n^*(s-)+\epsilon_n)}, 
\]
for all $s\ge 0$, where for each $\Phi \in L(U,V)$, we denoted by $\Phi^\dagger$ the adjoint of $\Phi$. Note, that the assumptions on $\drift_n$ and $\diff_n$, together with \eqref {eq.first-assumption-C_n},  guarantee that $\alpha_n$ is bounded by $K_1$, while $\beta_n$ and  $\beta_n^\dagger$ are bounded by $K_2$. For each \(t>0\), the semimartingale \(Z_n\) admits the representation

\[
Z_n(t)=\int_{(0,t]} (Z_n^*(s-)+\epsilon_n)\alpha_n(s)\,{\rm d}s+\int_{(0,t]} (Z_n^*(s-)+\epsilon_n)\beta_n(s)\,{\rm d}L(s).
\]
Applying the integration by parts formula \cite[Cor.\ 26.6]{met82}, Part (b) of
Lemma \ref{le.prop_int_proc} and Lemma \ref{le.quadratic_variation_of_integral}, we obtain
\begin{align}\label{eq.squared_integrator_control}
    &\int_{(0,t]}\frac{1}{\left(Z_n^*(s-)\right)^2+\epsilon_n^2}\;{\rm d}{\norm{Z_n}^2}(s)\nonumber\\
    &= 2\int_{(0,t]} \frac{\langle Z_n(s-), \cdot\rangle\,}{\left(Z_n^*(s-)\right)^2+\epsilon_n^2}\;{\rm d}Z_n(s)+\int_{(0,t]}\frac{1}{\left(Z_n^*(s-)\right)^2+\epsilon_n^2}\;\d[Z_n](s)\nonumber\\
    &= 2\left(\int_{(0,t]} \frac{\langle Z_n(s-), \alpha_n\rangle(Z_n^*(s-)+\epsilon_n)}{\left(Z_n^*(s-)\right)^2+\epsilon_n^2}\,{\rm d}s+ \int_{(0,t]} \frac{\langle \beta_n^\dagger Z_n(s-),\cdot \rangle(Z_n^*(s-)+\epsilon_n)}{\left(Z_n^*(s-)\right)^2+\epsilon_n^2}\,{\rm d}L(s)\right)\nonumber\\
    &\qquad\qquad+ \left[\int_{(0,\cdot]} \frac{(Z_n^*(s-)+\epsilon_n)}{\left((Z_n^*(s-))^2+\epsilon_n^2\right)^{1/2}}\beta_n\,{\rm d}L(s)\right](t).
\end{align}
Let $(\lambda_n)_{n \in \mathbb{N}}$ be a decreasing sequence of positive real numbers with $\lim_{n \rightarrow \infty}\lambda_n=0$. Using the elementary inequality $(a+b)^2\leq 2(a^2+b^2)$ for every $a,b\in\R$, all integrands on the right hand side of Equation \eqref{eq.squared_integrator_control} are bounded. Consequently, the dominated convergence theorem, or alternatively a direct estimate, yields
\begin{equation}\label{eq.dct_1}
    \lim_{n \rightarrow \infty}E\left[\sup_{t \in [0,T]}\left \vert\int_{(0,t]} \lambda_n \frac{\langle Z_n(s-), \alpha_n\rangle(Z_n^*(s-)+\epsilon_n)}{\left(Z_n^*(s-)\right)^2+\epsilon_n^2}\,{\rm d}s\right\vert \wedge 1\right]=0, 
\end{equation}
and Corollary \ref{co.bdd_stoch_convergence} implies
\begin{equation}\label{eq.dct_2}
    \lim_{n \rightarrow \infty}E\left[\sup_{t \in [0,T]}\abs{\int_{(0,t]} \lambda_n\frac{\langle \beta_n^\dagger Z_n(s-), \cdot\rangle(Z_n^*(s-)+\epsilon_n)}{\left(Z_n^*(s-)\right)^2+\epsilon_n^2}\,{\rm d}L(s)} \wedge 1\right]=0.
\end{equation}
Similarly, Corollary \ref{co.bdd_stoch_convergence} and Lemma \ref{le.cont_of_quad_variation} together guarantee
\begin{equation}\label{eq.dct_3}
    \lim_{n \rightarrow \infty}E\left[\sup_{t \in [0,T]}\abs{\left[\int_{(0,\cdot]} \lambda_n \frac{(Z_n^*(s-)+\epsilon_n)}{\left((Z_n^*(s-))^2+\epsilon_n^2\right)^{1/2}}\beta_n \,{\rm d}L(s)\right](t)} \wedge 1\right]=0.
\end{equation}
Combining Equations \eqref{eq.log_bound_on_ito_process}-\eqref{eq.dct_3}, we conclude that for every $t\in [0,T]$, 
\begin{align*}
    \lim_{n \rightarrow \infty}\lambda_n &\ln\left(\epsilon_n^{-2}\left(D_n^*(t)\right)^2+1\right)
   =0 \quad \text{in } L_P^0(\Omega,\mathbb{R}).
\end{align*}
Since the above convergence holds for an arbitrary sequence $(\lambda_n)_{n\in\N}$ decreasing to $0$, by Lemma \ref{le.bounded_in_prob}, the sequence $(\ln\left(\epsilon_n^{-2}(D_n^*(t))^2+1\right))_{n \in \mathbb{N}}$ of random variables is bounded in probability. Hence, we obtain by exponentiation  that $(\epsilon_n^{-2}(D_n^*(t))^2)_{n \in \mathbb{N}}$ is also bounded in probability. Since $\epsilon_n^{-2} \rightarrow \infty$, it follows  that $D_n^*(t) \rightarrow 0$ in probability. Since this argument holds for all $t\geq0$, we conclude that $D_n \rightarrow 0$ uniformly on compacts in probability. By the very definition of the sequence $(D_n)_{n \in \mathbb{N}}$ and Assumption \eqref{eq.first-assumption-C_n}, we deduce that $Y_n \rightarrow 0$ uniformly on compacts in probability.

To prove the result without assuming \eqref{eq.first-assumption-C_n}, fix $T>0$. Lemma \ref{le.ucp_consequence} implies the existence of a strictly positive sequence $(\epsilon_n)_{n\in\N}$ converging to $0$ and satisfying $P(C_n^\ast(T)>\epsilon_n)\to 0$ as $n\to \infty$. Define a sequence $(\tau_n)_{n \in \mathbb{N}}$ of stopping times by
\[\tau_n:=\inf\{t\geq 0:\norm{C_n(t)}>\epsilon_n\},\]
and a sequence $(\widetilde{Y}_n)_{n \in \mathbb{N}}$ of auxiliary processes by
\[\widetilde{Y}_n(t):={S_n(t)}\left(\int_{(0,t]} \drift_n(s) {\mathbb{1}_{(0,\tau_n]}}\,{\rm d}s+\int_{(0,t]} \diff_n(s){\mathbb{1}_{(0,\tau_n]}(s)}\,{\rm d}L(s)\right)+C_n(t)\mathbb{1}_{[0,\tau_n)}(t).\]
The first part of the proof, with $\tilde{\drift}_n = \drift_n\mathbb{1}_{[0,\tau_n]}$ and $\tilde{\diff}_n = \diff_n \mathbb{1}_{[0,\tau_n]}$, guarantees that $(\widetilde{Y}_n)_{n \in \mathbb{N}}$ converges to $0$ uniformly on compacts in probability. Since  
\[\lim_{n \rightarrow \infty}P(\tau_n\leq T)=\lim_{n \rightarrow \infty}P(C_n^*(T)>\epsilon_n)=0, \]
and $\widetilde{Y}_n(t)=Y_n(t)$ on  $\{\tau_n>t \}$,  we conclude  
\begin{align*}
	\lim_{n \rightarrow \infty}P\left(Y_n^*(T)>\epsilon\right)&\le \lim_{n \rightarrow \infty}\bigg(P\Big(\{Y_n^*(T)>\epsilon\} \cap \{\tau_n> T\}\Big)+P(\tau_n\leq T)\bigg)\\
	&\le \lim_{n \rightarrow \infty}P\left(\widetilde{Y}_n^*(T)>\epsilon\right)+ \lim_{n \rightarrow \infty}P(\tau_n \leq T)=0,
\end{align*}
which completes the proof.
\end{proof}

Given Equation \eqref{eq.equation}, we define a mapping $\Lambda:D_H\rightarrow D_H$ by
\begin{equation}\label{eq.fixed_point_operator}
    \Lambda(X)(t)=S(t)X_0+\int_{(0,t]} S(t-s)\drift(X(s))\,{\rm d}s+\int_{(0,t]} S(t-s)\diff(X(s-))\,{\rm d}L(s),
\end{equation}
where the fact that $\Lambda$ is well-defined follows from Lemma \ref{le.cont_of_lambda} below. To prove Theorem \ref{th.existence}, we will show that $\Lambda$ has a fixed point in $D_H$, that is, an $H$-valued adapted c\`adl\`ag process $X$ satisfying
\[
\Lambda(X)(t)= X(t)\qquad \text{$P$-a.s.\ for all }t\ge 0.
\]
Before proceeding with the fixed-point argument, we first establish that $\Lambda$ is well defined and examine some of its key properties in the following lemmas.
\begin{lemma}\label{le.cont_of_lambda}
  Equation  \eqref{eq.fixed_point_operator} defines a continuous  map $\Lambda:(D_H,\ducp) \rightarrow (D_H,\ducp)$. 
\end{lemma}

\begin{proof}
Since for $X\in D_H$ and fixed $t>0$, the process $(S(t-s)\drift(X(s)):\, s\in [0,t])$ is strongly measurable and Bochner integrable on $(0,t]$, 
while the process $(S(t-s)\diff(X(s-)):\, s\in [0,t])$ is predictable and has c\`agl\`ad paths, the integrals in the definition of $\Lambda(X)(t)$ exist; see also Remark~\ref{re.caglad_is_integrable}. 

To prove that $\Lambda(X)$ has c\`adl\`ag trajectories for each $X\in D_H$, we apply the dilation theorem \cite[Th. I.8.1]{NAFO} to conclude that there exists a Hilbert space $\Hat{H}$ such that $H \subseteq \Hat{H}$ with an isometric embedding $i:H\rightarrow \hat{H}$ and a group of unitary operators $(\Hat{S}_t)_{t \in \mathbb{R}}$ such that $S(t)=\pi\Hat{S}(t)i$ for all $t \geq 0$, where $\pi:\Hat{H}\rightarrow H$ denotes the orthogonal projection from $\hat{H}$ to $H$. Applying Part~(a) of Lemma~\ref{le.prop_int_proc} gives that, for every \(t>0\), 
\begin{align}\label{eq.dilation_of_lambda_operator}
&\Lambda(X)(t)\nonumber \\
&\; =\pi\hat{S}(t)iX_0+\int_{(0,t]} \pi\hat{S}(t-s)i\drift(X(s))\,{\rm d}s+\int_{(0,t]} \pi\hat{S}(t-s)i\diff(X(s-))\,{\rm d}L(s)\nonumber\\
&\;=\pi\hat{S}(t)iX_0+\pi\hat{S}(t)\int_{(0,t]} \hat{S}(-s)i\drift(X(s))\,{\rm d}s+\pi\hat{S}(t)\int_{(0,t]} \hat{S}(-s)i\diff(X(s-))\,{\rm d}L(s)\nonumber\\
&\;=\pi\hat{S}(t)\left(iX_0+\int_{(0,t]} \hat{S}(-s)i\drift(X(s))\,{\rm d}s+\int_{(0,t]} \hat{S}(-s)i\diff(X(s-))\,{\rm d}L(s)\right).
\end{align}
Since $\drift$ is continuous, the $\hat{H}$-valued Lebesgue integral process
\[\left(\int_{(0,t]} \hat{S}(-s)i\drift(X(s))\,{\rm d}s:t \geq 0\right)\]
is well-defined and has continuous paths. Since \(\diff\) is continuous, the process \(\bigl(\hat{S}(-s)i\,\diff(X(s-)) : s \ge 0\bigr)\) is
predictable with c\`agl\`ad paths, and is therefore integrable by Remark~\ref{re.caglad_is_integrable}. 
By Theorem~7.4 in \cite{BR}, the stochastic integral
\[
\left(
\int_{(0,t]} \hat{S}(-s)i\,\diff(X(s-))\,\mathrm{d}L(s)
: t \ge 0
\right)
\]
is an \(\hat{H}\)-valued c\`adl\`ag semimartingale.  
Hence, by \eqref{eq.dilation_of_lambda_operator}, we conclude that 
\(\Lambda(X) \in D_H\).

To establish continuity of $\Lambda:(D_H,\ducp) \rightarrow (D_H,\ducp)$, let  $(Y_n)_{n\in \mathbb{N}}$ be a sequence in $D_H$  converging to a process $Y$ in  $D_H$ and fix $T>0$. Applying the representation \eqref{eq.dilation_of_lambda_operator}, we obtain that
\begin{align}\label{eq.cont_lambda_split}
    &E\left[\sup_{t \in [0,T]} \norm{\Lambda(Y_n)(t)-\Lambda(Y)(t)}_H\wedge 1\right]\nonumber\\
    &\qquad\leq E\left[\sup_{t \in [0,T]} \norm{S(t)(Y_n(0)-Y(0))}_H\wedge 1\right]\nonumber\\
    &\qquad\quad+E\left[\sup_{t \in [0,T]} \norm{\pi\hat{S}(t)\int_{(0,t]} \hat{S}(-s)i\Big(\drift(Y_n(s))-\drift(Y(s))\Big)\,{\rm d}s}_H\wedge 1\right]\nonumber\\
    &\qquad\quad+E\left[\sup_{t \in [0,T]} \norm{\pi\hat{S}(t)\int_{(0,t]} \hat{S}(-s)i\Big(\diff(Y_n(s-))-\diff(Y(s-))\Big)\,{\rm d}L(s)}_H\wedge 1\right]\nonumber\\
    &\qquad:=C_{1,n}+C_{2,n}+C_{3,n}.
\end{align}
Since $S$ is a semigroup of contractions, it follows from Lebesgue's dominated convergence theorem that
\[\lim_{n \rightarrow \infty} C_{1,n}\leq \lim_{n \rightarrow \infty}E\left[\norm{Y_n(0)-Y(0)}_H\wedge 1\right]=0.\]
It follows from the dilation theorem and Lipschitz continuity of $\drift$ that
\begin{align*}
    C_{2,n} 
    &\leq E\left[ \int_{(0,T]} \norm{\hat{S}(-s)i\Big(\drift(Y_n(s))-\drift(Y(s))\Big)}_{\hat{H}}\,{\rm d}s\wedge 1\right]\\
    &\leq \max{\{c_{\drift} T, 1\}} E\left[\sup_{t \in [0,T]} \norm{Y_n(t)-Y(t)}_{H}\wedge 1\right],
\end{align*}
which establishes $\lim_{n \rightarrow \infty} C_{2,n}=0$.

For estimating the term $C_{3,n}$ in \eqref{eq.cont_lambda_split}, define the $L_2(U,\Hat{H})$-valued random variables 
\[
Z_n(t):= \hat{S}(-t)i\Big(\diff(Y_n(t-))-\diff(Y(t-))\Big) \qquad\text{for }t\ge 0. 
\] 
Lipschitz continuity of $\diff$ and the dilation theorem imply, for each $n\in\N$ and $T>0$, that
\begin{align}\label{eq.cont_ucp_integrand}
    E\left[\sup_{t \in [0,T]}\norm{Z_n(t)}\wedge 1\right]
    \leq \max\{c_{\diff},1\} E\left[\sup_{t \in [0,T]} \norm{Y_n(t-)-Y(t-)}_H\wedge 1\right], 
\end{align}
from which it follows that $(Z_n^*(T))_{n\in\N}$ converges to $0$ in probability. It follows that for every subsequence  $(C_{3,n_m})_{m \in \mathbb{N}}$  of $(C_{3,n})_{n \in \mathbb{N}}$ there exists a further subsequence $(C_{3,n_{m_k}})_{k \in \mathbb{N}}$ and a set $\Omega_1\subseteq \Omega$ with $P(\Omega_1)=1$, such that for each $\omega \in \Omega_1$, we have $Z_{n_{m_k}}^*(\omega,T) \rightarrow 0$. In particular, for all $\omega\in \Omega_1$ there exists a constant $c(\omega)<\infty$, such that $\sup_{k \in \mathbb{N}} Z_{n_{m_k}}^*(\omega,T) \leq c(\omega)$. Hence, for each fixed $\omega \in \Omega_1$, it follows from \cite[Le. 3.16]{BR} that there exists a constant $d(\omega)<\infty$, such that
\begin{align*}
     \sup_{k \in \mathbb{N}}\sup_{t \in [0,T]}\left( \zeta_L(Z_{n_{m_k}}(\omega,t))+\eta_L(Z_{n_{m_k}}(\omega,t))\right)<d(\omega).
\end{align*}
Thus, for each fixed $\omega \in \Omega_1$, we can apply Lebesgue's dominated convergence theorem with the constant in time dominating function $d(\omega)+1$, to conclude
\[\lim_{k\rightarrow \infty}m_L(Z_{n_{m_k}}(\omega))=0\qquad \text{for all }\omega \in \Omega_1.\]
As $P(\Omega_1)=1$, it follows that
\[\lim_{k\rightarrow \infty}E\left[m_L(Z_{n_{m_k}})\wedge 1\right]=0,\]
which, by Corollary~7.3 in~\cite{BR}, implies
 $\lim_{k\to\infty} d_T^{em}(I(Z_{n_{m_k}}),0)=0$. As convergence in $d_T^{em}$ implies convergence in $d_T^{ucp}$, we have
\begin{align*}
    \lim_{k \rightarrow \infty}C_{3,{n_{m_k}}}
        \leq \lim_{k \rightarrow \infty}E\left[\sup_{t \in [0,T]} \norm{\int_{(0,t]} Z_{n_{m_k}}(s) \,{\rm d}L(s)}_{\hat{H}}\wedge 1\right]=0.
\end{align*}
Since the subsequence $(C_{3,n_m})_{m \in \mathbb{N}}$ is arbitrary it follows that $C_{3,n}\to 0$ 
as $n\to\infty$. 

It follows from \eqref{eq.cont_lambda_split} that 
\[
\lim_{n\to\infty} E\left[\sup_{t \in [0,T]} \norm{\Lambda(Y_n)(t)-\Lambda(Y)(t)}_H\wedge 1\right]=0, 
\]
which establishes continuity of the map $\Lambda:(D_H,\ducp) \rightarrow (D_H,\ducp)$.
\end{proof}

\begin{remark}\label{re.lambda_representation_dilated}
	The dilation formula used in \eqref{eq.dilation_of_lambda_operator} in 
	Lemma~\ref{le.cont_of_lambda} yields an alternative representation of 
	\(\Lambda\). 
	Specifically, for every \(X \in D_H\) and \(t \ge 0\),
	\[
	\Lambda(X)(t)
	= \pi\hat{S}(t)\Biggl(
	iX_0
	+ \int_{(0,t]} \hat{S}(-s)i\,\drift(X(s))\,\mathrm{d}s
	+ \int_{(0,t]} \hat{S}(-s)i\,\diff(X(s-))\,\mathrm{d}L(s)
	\Biggr).
	\]
	This representation will be crucial later when evaluating \(\Lambda(X)\) at stopping 
	times, as it ensures the existence of a particularly convenient expression for the 
	resulting random variables.
	
	 The dilation theorem from \cite{NAFO} has a broad range of applications. One of the first uses in the context of SPDEs appears in \cite{Hausenblas-Seidler}.
\end{remark}

\begin{lemma}\label{le.Lambda-invariant}
Let $X$ and $Y$ be processes in $D_H$ and let $\tau$ be a stopping time.
If $X(t)=Y(t)$ for all $t<\tau$, then for every $T\geq0$ we obtain $\Lambda(X)(\tau \wedge T)=\Lambda(Y)(\tau \wedge T)$.
\end{lemma}
\begin{proof} 
By Remark \ref{re.lambda_representation_dilated}, 
it suffices to show that for every $T\geq 0$,
\[\pi\hat{S}(\tau \wedge T)\int_{(0,\tau \wedge T]} \hat{S}(-s)i\diff(X(s-))\,{\rm d}L(s)=\pi\hat{S}(\tau \wedge T)\int_{(0,\tau \wedge T]} \hat{S}(-s)i\diff(Y(s-))\,{\rm d}L(s).\]
By part (c) of Lemma \ref{le.prop_int_proc}, we have
\[\int_{(0,\tau \wedge T]} \hat{S}(-s)i\diff(X(s-))\,{\rm d}L(s)=\int_{(0,T]} \mathbb{1}_{(0,\tau]}(s) \hat{S}(-s)i\diff(X(s-))\,{\rm d}L(s).\]
Since  $X(t)=Y(t)$ for all $t<\tau$, and the left limits are determined by the behaviour of the process  strictly before the time point, it follows that $X(t-)=Y(t-)$ for all $t\in (0,\tau]$. Consequently,
\[\int_{(0,T]} \mathbb{1}_{(0,\tau]}(s) \hat{S}(-s)i\diff(X(s-))\,{\rm d}L(s)=\int_{(0,T]} \mathbb{1}_{(0,\tau]}(s) \hat{S}(-s)i\diff(Y(s-))\,{\rm d}L(s).\]
Applying part~(c) of Lemma~\ref{le.prop_int_proc} once more yields
\[\int_{(0,T]} \mathbb{1}_{(0,\tau]}(s) \hat{S}(-s)i\diff(Y(s-))\,{\rm d}L(s)=\int_{(0,\tau\wedge T]} \hat{S}(-s)i\diff(Y(s-))\,{\rm d}L(s),\]
which completes the proof.
\end{proof}

The following is immediate from the c\`adl\`ag paths of the underlying processes. For later reference and completeness, we state it here and provide a proof.
\begin{lemma}\label{lemma:stopping_times_increase_1}
Assume that  $Y$ is an $H$-valued c\`adl\`ag adapted process. 
If  $\sigma_0\colon \Omega \rightarrow [0,\infty]$ is a stopping time with $Y(\sigma_0)=0$ on $\{\sigma_0<\infty\}$ and 
 $\sigma_1 := \inf\{t\geq \sigma_0 \colon  \norm{Y(t)} > \epsilon\}$ for some $\epsilon>0$, then 
$P(\{\sigma_1 = \sigma_0\} \cap \{ \sigma_0<\infty\})=0$.
\end{lemma}

\begin{proof}
For each fixed $\omega\in \{\sigma_1 = \sigma_0\}\cap\{\sigma_0 < \infty\}$ there exists a sequence $(t_k)_{k\in \N}$ in $(\sigma_0(\omega),\infty)$ decreasing to $ \sigma_0(\omega)$ and satisfying $\norm{Y(\omega, t_k)}>\frac{\epsilon}{2}$ for all $k\in\N$, due to right-continuity of $Y$. On the other hand, since $Y(\omega, \sigma_0(\omega))=0$, right-continuity also implies  $\lim_{k\to\infty }Y(\omega,t_k) =0$, which yields a contradiction.
\end{proof}

The main ingredient of our proof of Theorem \ref{th.existence} is the following result, which we will interpret later as a pathwise approximation of the solution of \eqref{eq.equation}. 
\begin{proposition}\label{pro.approx_solution}
    Let $\Lambda$ be the function defined in  \eqref{eq.fixed_point_operator}. Then for every $\epsilon>0$, there exists a process $X \in D_H$ satisfying, for all $\omega \in \Omega$, that
\[\sup_{t\geq0}\norm{X(\omega,t)-\Lambda(X)(\omega,t)}\le \epsilon.\]
\end{proposition}

\begin{proof} (The proof is adapted from \cite{Lowther}). 
Let $\epsilon >0$ be fixed. Define recursively a sequence of stopping times $(\tau_n)_{n \geq 0}$ and processes $(X_n)_{n \geq 0}$ by setting $\tau_0=0$ and $X_0(t)=X_0$ for all $t\ge 0$, and, for all $n\geq 0$, by 
    \[\tau_{n+1}=\inf\{t\geq \tau_n:\;\norm{X_{n}(t)-\Lambda(X_{n})(t)}\geq \epsilon\},\]
    and
    \begin{equation}\label{eq.X_n}
    X_{n+1}(t)=
    \begin{cases*}
      X_n(t) & if $t<\tau_{n+1}$, \\
      \Lambda(X_n)(\tau_{n+1}) & if $t\geq \tau_{n+1}$. 
    \end{cases*}
  \end{equation}
Since the sequence $(\tau_n)_{n \geq 0}$ of stopping times is non-decreasing, it has a possibly infinite limit  $\tau$. As $X_{n+1}(t)=X_n(t)$ for all $t<\tau_{n+1}$ and all $n\in\N$,  we can define a c\`adl\`ag process $X$ on the random interval $[0,\tau)$ by
\begin{equation}\label{eq.X}
    X(t)=X_n(t) \quad \text{for}\;t<\tau_{n+1}.
\end{equation}
The very definition of the sequence $(X_n)_{n \in \mathbb{N}}$ and Lemma \ref{le.Lambda-invariant} together guarantee that $\norm{X(t)-\Lambda(X)(t)}<\epsilon$ for all  $t\in [0,\tau)$. Hence, in order to finish the proof, it suffices to show that $P(\tau=\infty)=1$. To accomplish this, we proceed in steps.

\textit{Step 0:} We prove that, for all $n\geq 0$, 
\[P(\{\tau_{n+1} = \tau_{n}\} \cap \{ \tau_{n}<\infty\})=0.\]
To this end, define $Y(t)=X_{n}(t)-\Lambda(X_{n})(t)$ for $t\ge  0$. Lemma \ref{le.Lambda-invariant} implies $Y(\tau_n)=0$ since $X_{n-1}=X_n$ on $\{t< \tau_n\}$. Consequently, applying Lemma \ref{lemma:stopping_times_increase_1} with $\sigma_0=\tau_n$ and $\sigma_1=\tau_{n+1}$ establishes the claim.

\textit{Step 1:} Next, we prove that $P(X^*(T \wedge \tau)<\infty)=1$  for each fixed $T \geq 0$. 
The dilation theorem \cite[Th. I.8.1]{NAFO} guarantees that there exists a Hilbert space $\Hat{H}$ such that $H \subseteq \Hat{H}$ with an isometric embedding $i:H\rightarrow \hat{H}$ and a group of unitary operators $(\Hat{S}_t)_{t \in \mathbb{R}}$ such that $S(t)=\pi\Hat{S}(t)i$ for all $t \geq 0$, where $\pi:\Hat{H}\rightarrow H$ denotes the orthogonal projection from $\hat{H}$ to $H$. Remark \ref{re.lambda_representation_dilated} implies that for each $t>0$ we have 
\begin{align*}
    &X^{\tau_n \wedge T}(t)\\
    &=\Big(\Lambda\left(X^{\tau_n\wedge T}\right)(t)-\Lambda(0)(t)\Big)+ \Lambda(0)(t)+\Big(X^{\tau_n\wedge T}(t)-\Lambda\left(X^{\tau_n\wedge T}\right)(t)\Big)\\
    &=\pi\hat{S}(t)\bigg(\int_{(0,t]} \Hat{S}(-s)i\left(\drift\left(X^{\tau_n\wedge T}(s)\right)-\drift(0)\right)\,{\rm d}s\\
    &\qquad\qquad\qquad\qquad\qquad\qquad\qquad+\int_{(0,t]} \Hat{S}(-s)i\left(\diff\left(X^{\tau_n \wedge T}(s-)\right)-\diff(0)\right)\,{\rm d}L(s)\bigg)\\
    &\qquad\qquad+\pi\hat{S}(t)\left(ix_0+\int_{(0,t]} \Hat{S}(-s)i\drift(0)\,{\rm d}s+\int_{(0,t]} \Hat{S}(-s)i\diff(0)\,{\rm d}L(s)\right)\\
    &\qquad\qquad+\Big(X^{\tau_n \wedge T}(t)-\Lambda\left(X^{\tau_n \wedge T}\right)(t)\Big).
\end{align*}
Let $(\lambda_n)_{n\in \mathbb{N}}$ be a sequence  of real numbers converging to $0$. Multiplying both sides by 
$\lambda_n$ results in the identity
\begin{equation}\label{eq.approx_ucp_2}
	\lambda_n X^{\tau_n \wedge T}(t)= \pi \Hat{S}(t)\left(\int_{(0,t]} A_n(s)\, {\rm d}s + \int_{(0,t]} B_n(s)\, {\rm d}L(s)\right) + C_n(t),
\end{equation}
where
\begin{align*}
    A_n(s) &:= \lambda_n \Hat{S}(-s)i\Big(\drift\left(X^{\tau_n\wedge T}(s-)\right)-\drift(0)\Big);\\
    B_n(s) &:= \lambda_n \Hat{S}(-s)i\Big(\diff\left(X^{\tau_n \wedge T}(s-)\right)-\diff(0) \Big);\\
    C_n(t) &:= \lambda_n\pi\hat{S}(t)\left(ix_0+\int_{(0,t]} \Hat{S}(-s)i\drift(0)\,{\rm d}s+\int_{(0,t]} \Hat{S}(-s)i\diff(0)\,{\rm d}L(s)\right)\\
    &\qquad+\lambda_n\Big(X^{\tau_n \wedge T}(t)-\Lambda\left(X^{\tau_n \wedge T}\right)(t)\Big).
\end{align*}
Letting $Y_n(t):=\lambda_n X^{\tau_n \wedge T}(t)$, Lipschitz continuity  implies
\begin{align*}
    \norm{A_n(s)}
    \leq \lambda_n  c_{\drift} \norm{X^{\tau_n\wedge T}(s-)}
    = c_{\drift} \norm{Y_n(s-)},
\end{align*}
where $c_{\drift}$ is the Lipschitz constant of $\drift$. By a completely analogous computation for $B_n$, and using that the semigroup $\Hat{S}$ is strongly continuous, we arrive at the estimates
\begin{equation}
A_n^*(t)\leq c_{\drift} Y_n^*(t-) \qquad \text{ and }\qquad  B_n^*(t)\leq c_{\diff} Y_n^*(t-). 
\end{equation}
Since by the definition of $X$ we have 
\[\norm{X^{\tau_n \wedge T}(t)-\Lambda\left(X^{\tau_n \wedge T}\right)(t)}<\epsilon \qquad \text{for all }\, t\geq 0, n\in\N,\]
we obtain that the sequence $(C_n)_{n \in \mathbb{N}}$ converges to $0$ uniformly on compacts in probability. By Proposition \ref{pro.ucp_convergence}, the sequence $(Y_n)_{n\in \mathbb{N}}$  converges to $0$ uniformly on compacts in probability, which further implies that the sequence $(\lambda_n X^{\tau_n \wedge T})_{n\in \mathbb{N}}$ of $H$-valued processes also converges to $0$ uniformly on compacts in probability.  Lemma \ref{le.bounded_in_prob} shows that the collection $(X^*(\tau_n \wedge T))_{n \in \mathbb{N}}$ of random variables is bounded in probability, which means 
\[\lim_{m \rightarrow \infty}\sup_{n \in \mathbb{N}}P(X^*(\tau_n \wedge T)>m)=0.\]
Standard arguments show
\begin{align*}
    P\left(X^*(T \wedge \tau)=\infty\right)
    &=P\left(\bigcap_{m=1}^\infty \bigcup_{n=1}^\infty\left\{ X^*\left(T \wedge \tau_n\right) > m\right\}\right)\\
    &= \lim_{m\rightarrow \infty}\sup_{n \in \mathbb{N}}P\Big(X^*(T\wedge \tau_n)>m\Big)=0.
\end{align*}

\textit{Step 2:}  We complete the proof  by establishing
$P\left(\tau =\infty \right)=1$. 
Fix \(T>0\). Since \(X(t)=X_n(t)\) for all \(t<\tau_{n+1}\) and every \(n\in\mathbb N\), Lemma~\ref{le.Lambda-invariant} yields
\begin{align*}
 X(\omega,\tau_n)
= X_{n}(\omega,\tau_n)
=\Lambda(X_{n-1})(\omega,\tau_n)
=\Lambda(X)(\omega,\tau_n) \qquad {\text{for all } \omega \in \{\tau\leq T\}}.
\end{align*}
Applying Remark \ref{re.lambda_representation_dilated} together with part (c) of Lemma \ref{le.prop_int_proc}, we obtain
\begin{align}\label{eq.aux}
	\begin{split}
{\mathbb{1}_{\{\tau \leq T\}}}\,X(\tau_n)
&=	\mathbb{1}_{\{\tau \leq T\}}\pi \Hat{S}(\tau_n)\bigg(ix_0+\int_{(0,T]}  \mathbb{1}_{(0,\tau_n]}(s)\Hat{S}(-s)i\drift (X(s))\,{\rm d}s\\
&\qquad\qquad+\int_{(0,T]} \mathbb{1}_{(0,\tau_n]}(s)\Hat{S}(-s)i\diff(X(s-))\,{\rm d}L(s)\bigg). 
\end{split}
\end{align}
By Step 1 and Lipschitz continuity, the processes
\[
\bigl(\mathbb 1_{(0,\tau\wedge T]}(t)\,\hat S(-t)i\,\drift(X(t)):\,t\ge0\bigr)
\quad\text{and}\quad
\bigl(\mathbb 1_{(0,\tau\wedge T]}(t)\,\hat S(-t)i\,\diff(X(t-)):\,t\ge0\bigr)
\]
are integrable for every \(T\ge0\); the former in the Lebesgue sense, the latter stochastically with respect to \(L\) by Remark~\ref{re.caglad_is_integrable}. Moreover, the first of these processes dominates the integrand of the first integral in \eqref{eq.aux}, while the $m_L$-functional of the second process dominates the $m_L$-functional of the integrand in the second integral in \eqref{eq.aux}.  Hence, by Lebesgue’s dominated convergence theorem and Theorem~\ref{th_dom_stoch_conv}, we obtain
\begin{align*}
	&\lim_{n\to\infty} {\mathbb{1}_{\{\tau \leq T\}}}\,X(\tau_n)\\
	&\qquad\qquad=
	{\mathbb{1}_{\{\tau \leq T\}}}\,\pi \hat S(\tau)\Bigg(
	ix_0
	+ \int_{(0,T]}
	\mathbb 1_{(0,\tau\wedge T]}(s)\,\hat S(-s)i\,\drift(X(s))\,\mathrm ds \\
	&\qquad\qquad\qquad\qquad\qquad\qquad\qquad\qquad
	+ \int_{(0,T]}
	\mathbb 1_{(0,\tau\wedge T]}(s)\,\hat S(-s)i\,\diff(X(s-))\,\mathrm dL(s)
	\Bigg).
\end{align*}
In particular, the sequence \(\bigl({\mathbb{1}_{\{\tau \leq T\}}}\,X(\tau_n)\bigr)_{n\in\mathbb N}\) converges in probability.

Since \(\tau_n<\tau_{n+1}\) and $X_n$ is constant on $[\tau_n,\infty)$ for all \(n\in\mathbb N\), it follows
 from the definition of \(X\) and \(X_n\) that
\[
\|X(\tau_{n+1})-X(\tau_n)\|\
=\|X_{n+1}(\tau_{n+1})-X_n(\tau_n)\|
=\|\Lambda(X_{n})(\tau_{n+1})-X_n(\tau_{n+1})\|
\ge\varepsilon
\]
$P$-a.s.\ on $\{\tau\le T\}$.
Consequently,
\begin{align*}
(1 \wedge \epsilon)\,P(\tau\le T)
	&\le
	E\!\left[\bigl\|\mathbb 1_{\{\tau\le T\}}\bigl(X(\tau_{n+1})-X(\tau_n)\bigr)\bigr\|\wedge1\right]. 
\end{align*}
Since \(\bigl({\mathbb{1}_{\{\tau \leq T\}}}\,X(\tau_n)\bigr)_{n\in\mathbb N}\) is Cauchy in \(L^0_P(\Omega;H)\), the right side converges to zero as \(n\to\infty\), and hence \(P(\tau\le T)=0\). 

Since \(P(\tau\le T)=0\) for all $T>0$, we conclude \( P(\tau=\infty)=1\).  
%
\end{proof}

\begin{lemma}\label{le.conv_to_solution}
    Let $(X_n)_{n \in \mathbb{N}}$ be a sequence of processes in $D_H$ such that 
    \begin{equation*}\lim_{n\rightarrow \infty}\ducp(X_n,\Lambda(X_n))= 0. 
    \end{equation*}
Then, there exists a process $X \in D_H$ so that 
$$ \lim_{n\rightarrow \infty}\ducp(X_n,X)= 0$$ and $X=\Lambda(X)$.
\end{lemma}

\begin{proof}
It suffices to establish that there exists a process $X \in D_H$ such that the following convergences hold:
\[\Lambda(X)\stackrel{{\rm(i)}}{=}\lim_{n \rightarrow \infty}\Lambda(X_n)\stackrel{{\rm(ii)}}{=}\lim_{n \rightarrow \infty}X_n\stackrel{{\rm(iii)}}{=}X,\]
where the limits are  in the topology of uniform convergence in probability on compact intervals. The fact that ${\rm(ii)}$ holds is a direct consequence of our assumptions, provided at least one of the two limits exists. Hence, it remains only to establish that there exists a process $X \in D_H$ such that ${\rm(i)}$ and ${\rm(iii)}$ both hold.

For each $m,n \in \mathbb{N}$ we introduce the processes
    \begin{align*}
        D_{m,n}&:=X_m-X_n-\Lambda(X_m)+\Lambda(X_n)\\
        A_{m,n}&:=\drift(X_m)-\drift(X_n)\\
        B_{m,n}&:=\diff(X_m)-\diff(X_n).
    \end{align*}
It follows that 
\begin{align*}
    X_m(t)-X_n(t)&= \Lambda(X_m)(t)-\Lambda(X_n)(t)+D_{m,n}(t)\\
    &=\int_{(0,t]} S(t-s)A_{m,n}(s-)\,{\rm d}s+\int_{(0,t]} S(t-s)B_{m,n}(s-)\,{\rm d}L(s)+D_{m,n}(t),
\end{align*}
where $\ducp(D_{m,n},0)\rightarrow 0$ as $m,n \rightarrow \infty$ and 
\begin{align*}
    A_{m,n}^*(s-)&\leq c_\drift (X_m-X_n)^*(s-)\\
    B_{m,n}^*(s-)&\leq c_\diff (X_m-X_n)^*(s-).
\end{align*}
Using an argument analogous to that in Proposition~\ref{pro.approx_solution}, the dilation theorem \cite[Th.~I.8.1]{NAFO} enables us to apply Proposition~\ref{pro.ucp_convergence} and conclude that $X_m-X_n \rightarrow 0$ in probability uniformly on compacts as $m,n \rightarrow \infty$. Since $D_H$ is complete with respect to convergence in probability uniformly on compact intervals, there exists a process $X \in D_H$ such that $X_n \rightarrow X$ in probability uniformly on compacts  as $n \rightarrow \infty$. This establishes the existence of a process $X \in D_H$ for which ${\rm (iii)}$ holds. Finally, combining ${\rm (iii)}$ with Lemma \ref{le.cont_of_lambda} yields  $\Lambda(X_n)\rightarrow \Lambda(X)$ in probability uniformly on compacts as $n \rightarrow \infty$,  which establishes ${\rm (i)}$.
\end{proof}

Lemma \ref{le.conv_to_solution} allows us to approximate the solution of the evolution equation \eqref{eq.equation}. Indeed, the proof of Proposition \ref{pro.approx_solution} explicitly constructs the processes $X_n$ pathwise. Lemma \ref{le.conv_to_solution} then establishes that these processes converge to the solution of \eqref{eq.equation}. This reasoning forms the basis of the following proof of our main result, Theorem \ref{th.existence}.

\begin{proof}[Proof of Theorem \ref{th.existence}]  By Proposition \ref{pro.approx_solution}, for each $n \in \mathbb{N}$ there exists a process $X_n \in D_{H}$ such that $\sup_{t \geq 0}\norm{X_n(t)-\Lambda(X_n)(t)}< 1/n$. Hence, Lemma \ref{le.conv_to_solution} implies that there exists a process $X \in D_H$ such that $X=\Lambda(X)$ and $X_n \rightarrow X$ in probability uniformly on compact intervals as $n \rightarrow \infty$. Thus, $X$ is a fixed point of $\Lambda$ and hence a mild solution of Equation \eqref{eq.equation}.

To establish uniqueness, assume that $X \in D_H$ and $Y \in D_H$ are both mild solutions of Equation \eqref{eq.equation}, that is, $\Lambda(X)=X$ and $\Lambda(Y)=Y$. Define a sequence $(Z_n)_{n \in \mathbb{N}}$ in $D_H$ by
\begin{align*}
    Z_n=
\begin{cases}
X \quad \text{if }\,n \, \text{ is odd};\\
Y \quad \text{if }\,n \, \text{ is even}.
\end{cases}
\end{align*}
It follows from the definition of $(Z_n)_{n \in \mathbb{N}}$ that $Z_n-\Lambda(Z_n)\rightarrow 0$ in probability uniformly on compact intervals as $n \rightarrow \infty$. Thus, by Lemma \ref{le.conv_to_solution}, $(Z_n)_{n \in \mathbb{N}}$ converges in probability uniformly on compact intervals, hence subsequences $(Z_{2n})_{n \in \mathbb{N}}$ and $(Z_{2n-1})_{n \in \mathbb{N}}$ converge in probability uniformly on compact intervals to the same limit. Consequently, we must have that $X=Y$, that is $X$ and $Y$ are indistinguishable. 
\end{proof}

\end{document}